\documentclass[11pt]{amsart}

\usepackage{mathrsfs}

\usepackage{amsmath}

\usepackage{amssymb}
\usepackage{amsfonts}

\usepackage{amscd}
\usepackage{bbm}

\usepackage[usenames, dvipsnames]{color}
\usepackage[all]{xy}

\newtheorem{proposition}{Proposition}[section]

\newtheorem{lemma}[proposition]{Lemma}
\newtheorem{definition}[proposition]{Definition}
\newtheorem{theorem}[proposition]{Theorem}

\newtheorem{corollary}[proposition]{Corollary}

\newtheorem{example}[proposition]{Example}
\newtheorem{remark}[proposition]{Remark}

\newtheorem{lemma-definition}[proposition]{Lemma-Definition}

\newtheorem{question}[proposition]{Question}

\numberwithin{equation}{section}

\newcommand{\Zz}{\mathbb{Z}}

\newcommand{\Pp}{\mathbb{P}}
\newcommand{\Rr}{\mathbb{R}}

\newcommand{\Hom}{\operatorname{Hom}}

\newcommand{\gr}{\operatorname{gr}}

\newcommand{\la}{\langle}
\newcommand{\ra}{\rangle}

\newcommand{\Oo}{\mathcal{O}}

\def\Coh{\mathop{\mathrm{coh}}\nolimits}

\def\dim{\mathop{\mathrm{dim}}\nolimits}

\def\Ext{\mathop{\mathrm{Ext}}\nolimits}

\def\Hom{\mathop{\mathrm{Hom}}\nolimits}

\def\min{\mathop{\mathrm{min}}\nolimits}

\def\perf{\mathop{\mathrm{perf}}}

\def\supp{\mathop{\mathrm{supp}}}

\def\Star{\mathrm{Star}}

\def\P{\ensuremath{\mathbb{P}}}

\def\Z{\ensuremath{\mathbb{Z}}}

\def\DD{\ensuremath{\mathcal D}}

\def\OO{\ensuremath{\mathcal O}}

\def\TT{\ensuremath{\mathcal T}}

\usepackage{hyperref}

\def\bD{\mathbf D}
\def\bT{\mathbf T}
\def\Aa{\mathbb A}
\def\Gm{\mathbf{G}_{\mathbf m}}

\def\kk{\mathbf k}
\def\aa{ \bar{\mathsf a}}
\def\pp{ \bar{\mathsf p}}
\def\qq{ \bar{\mathsf q}}
\def\rr{ \bar{\mathsf r}}
\def\ee{ \bar{\mathsf \epsilon}}
\def\00{ \bar{\mathsf 0}}

\def\Sn{S_n}
\def\Snt{S_{n, \theta}}
\def\bp{\mathbf p}
\def\bq{\mathbf q}

\def\cchi{ \bar{\chi}}

\def\zz{ \bar{\mathsf z}}

\def\tA{\widetilde{\mathcal A}}
\def\tB{\widetilde{\mathcal B}}
\def\tE{\widetilde{\mathcal E}}

\def\tM{\widetilde{\mathcal M}}
\def\tS{\widetilde{\mathcal S}}
\def\tSn{\widetilde{\mathcal S}_n}
\def\tSnt{\widetilde{\mathcal S}_{n,\theta}}

\def\Mod{\operatorname{Mod}\!\operatorname{--}\!}
\def\Qcoh{\operatorname{Qcoh}}
\def\op{\circ}
\def\GrL{\operatorname{Gr}^{L}\mkern-5mu\operatorname{--}\!}
\def\GrZ{\operatorname{Gr}^{\Z^n}\mkern-6mu\operatorname{--}\!}
\def\grL{\operatorname{gr}^{L}\mkern-5mu\operatorname{--}\!}

\def\oM{\operatorname{M}}

\def\perf{\mathcal{P}\!\mathit{erf}\!\operatorname{--}\!}

\def\lto{\longrightarrow}

\def\yY{\operatorname{Y}}

\def\xX{\operatorname{X}}

\def\kK{\operatorname{K}^{\cdot}}
\def\kKv{\operatorname{K}^{\cdot\vee}}

\def\FF{\ensuremath{\mathcal O}}

\def\js{j_{\Sigma}}

\begin{document}

\title[]{Equivariant exceptional collections on smooth toric stacks}

\begin{abstract}
We study the bounded derived categories of torus-equivariant coherent sheaves on smooth toric varieties and Deligne-Mumford stacks.
We construct and describe full exceptional collections in these categories.
We also observe that these categories depend only on the PL homeomorphism type of the corresponding simplicial complex.
\end{abstract}

%\date{\today}

\author[]{Lev  Borisov}
\address{Department of Mathematics\\
Rutgers University\\
110 Frelinghuysen Rd.\\
Piscataway, NJ 08854, USA}
\email{borisov@math.rutgers.edu}

\author[]{Dmitri Orlov}
\address{Algebraic Geometry Department, Steklov Mathematical Institute RAS,
8 Gubkin str., Moscow 119991, RUSSIA}
\email{orlov@mi.ras.ru}

\dedicatory{Dedicated to the blessed memory of our late advisor Vasily Alekseevich Iskovskikh on the occasion of his 80th birthday}


\maketitle

\section{Introduction}
Smooth toric varieties and Deligne-Mumford stacks are among the more accessible examples of algebro-geometric objects that can be used to build intuition for more general settings. In this paper we initiate the study of the bounded derived categories of torus-equivariant coherent sheaves on these varieties and stacks. Specifically, let $\Pp_\Sigma$ be a smooth toric DM stack or variety with the maximal torus $\bT_{\Pp_\Sigma}.$ Our main object of interest is the bounded derived category
$\bD^b(\Coh[\Pp_\Sigma/\bT_{\Pp_\Sigma}])$  of the Artin stack $[\Pp_\Sigma/\bT_{\Pp_\Sigma}]$ or equivalently the bounded derived category of $\bT_{\Pp_\Sigma}$\!--equivariant sheaves on $\Pp_\Sigma.$

%\smallskip
Our main results are the following.

We first consider the important case of $[\Aa^n/\bT]$ for $\bT=\Gm^n.$ We construct a strong full exceptional collection in  $\bD^b(\Coh[\Aa^n/\bT])$ that consists of shifts of equivariant line bundles on torus strata.
This leads to an interesting new description of this category.

\smallskip
\noindent
{\bf Theorem \ref{main1.5}.}\emph{
The derived category $\bD^b(\Coh[\Aa^n/\bT])$ of $\bT$\!--equivariant coherent sheaves on $\Aa^n$ possesses a strong full exceptional collection of objects $E_{a}$ indexed by $a=(a_1,\ldots,a_n)\in\Zz^n$
ordered lexicographically.
The morphism spaces are described as follows.
\[
\Hom(E_a,E_b)=
\left\{
\begin{array}{ll}
\kk,&{\rm if~ for~all~}i {\rm~either~} b_i-a_i\in \{0,1\} {\rm ~or~} (b_i,a_i) = (1,-1),\\
0,&{\rm otherwise.}
\end{array}
\right.
\]
The nontrivial spaces $\Hom(E_a,E_b)$ come equipped with a distinguished generator $\rho_{a,b}.$
The composition of morphisms $\Hom(E_a,E_b)\otimes \Hom(E_b,E_c)\to \Hom(E_a,E_c)$ is nontrivial if and only if all three spaces are nontrivial. In these cases $\rho_{b,c}\circ\rho_{a,b}=\rho_{a,c}.$
}

\smallskip
More generally, we consider $\bT$\!--invariant open subsets $U_\Sigma$ of $\Aa^n$ and
the corresponding Artin stacks $[U_\Sigma/\Gm^n].$ Every such  $U_\Sigma$  is defined in terms of a simplicial complex $\Sigma$ on  $[n]:=\{1,\ldots,n\},$ i.e. an inclusion-closed collection of subsets of $[n].$ A point $\{z_1,\ldots, z_n\}\in \Aa^n$ lies in $U_\Sigma$ if and only if there exists an element $\sigma\in\Sigma$ such that $z_i\neq 0$ for all $i\not\in\sigma.$ The open set $\Aa^n$ corresponds to $\Sigma$ that consists of all subsets of $[n].$

The analog of Theorem \ref{main1.5} in this more general setting is the following.

\smallskip
\noindent
{\bf Theorem \ref{main2.5}.}\emph{
The derived category $\bD^b(\Coh[U_\Sigma/\bT])$ of $\bT$\!--equivariant coherent sheaves on $U_\Sigma$ possesses a strong full exceptional collection of objects $E_{a}$ indexed by $a=(a_1,\ldots,a_n)\in\Zz^n$ such that the set of $i$ for which $a_i\neq 0$ lies in $\Sigma.$
The objects $E_a$ are ordered lexicographically.
The morphism space $\Hom(E_a,E_b)$ is isomorphic to $\kk$ if and only if for all $i$ either $b_i-a_i\in \{0,1\}$ or $(b_i,a_i) = (1,-1)$ and, moreover, the set of $i$ for which $a_i$ or $b_i$ are nonzero lies in $\Sigma.$ All other morphism spaces are zero. 
The nontrivial spaces $\Hom(E_a,E_b)$ come equipped with a distinguished generator $\rho_{a,b}.$
The composition of morphisms $\Hom(E_a,E_b)\otimes \Hom(E_b,E_c)\to \Hom(E_a,E_c)$ is nontrivial if and only if all three spaces are nontrivial. In these cases $\rho_{b,c}\circ\rho_{a,b}=\rho_{a,c}.$
}

\smallskip
The general case of quotients of smooth toric DM stacks $\P_\Sigma$ by  their maximal tori $\bT_{\Pp_\Sigma}$ is reduced to the setting of  Theorem \ref{main2.5} because every Artin stack $[\Pp_\Sigma/\bT_{\Pp_\Sigma}]$ 
is isomorphic to the Artin stack $[U_{\Sigma}/\bT].$ 
In particular, our construction applies to arbitrary smooth toric varieties.

It is worth mentioning that we were motivated by the work of Kawamata \cite{Kawamata} in the non-equivariant setting,
but our construction is very explicit, unlike the inductive description of the exceptional collection in \cite{Kawamata}. 
The equivariant setting allows for a much wider range of exceptional objects, such as the structure sheaves of torus invariant points, so it is perhaps not surprising that the equivariant category has a much nicer description than its non-equivariant counterpart. 

We briefly comment on the proofs of Theorems \ref{main1.5} and \ref{main2.5} which are contained in Sections \ref{sec2} and \ref{sec3} respectively. The objects $E_a$ are certain cohomological  shifts of equivariant twists of the (pushforwards of) structure sheaves of the closures of torus orbits given by 
$\bigcap_{i\in I}\{z_i=0\}$ for subsets $I\subseteq [n].$ We calculate  the $\Ext$ spaces between these sheaves using standard techniques of toric geometry, such as Koszul and \v{C}ech complexes. The exceptional collection is  proved to be full by showing that it generates arbitrary equivariant line bundles which in turn generate the whole derived category.
We try to keep our paper self-contained whenever possible  in hope of making it accessible to a wider audience.

We proceed in Section \ref{sec4} to investigate to
what extent the derived category  $\bD^b(\Coh[\Pp_\Sigma/\bT_{\Pp_\Sigma}])$  depends on $\Pp_\Sigma.$
This can be viewed as a study of a categorical invariant of simplicial complexes which
associates $\bD^b(\Coh[U_\Sigma/\bT])$ to $\Sigma.$ 

To describe our main result in this area,
recall that to any abstract simplicial complex $\Sigma$ one can associate a topological space $|\Sigma|$ called the polyhedron of $\Sigma$ (it is also called the geometric realization of $\Sigma$\!). It can be constructed by taking a copy $\Delta^{k-1}(\sigma)$ of a standard topological $(k-1)$-simplex for each $\sigma\in\Sigma$ with $|\sigma|=k,$ and gluing them together according to the face relations encoded in $\Sigma.$
More precisely, we have
\[
|\Sigma|=\Big\{x\in\Rr^n_{\ge 0}\; \Big|\; \sum_{i=1}^{n} x_i=1,\; \text{and}\quad \supp x\in\Sigma\Big\},
\]
where $\Rr^n_{\ge 0}\subset\Rr^n$ is the closed positive orthant and $\supp x=\{1\le i\le n\; |\; x_i\ne 0\}\subseteq\{1,\dots,n\}.$
In other words, the topological space $|\Sigma|$ consists of all points of  $\Rr^n_{\ge 0}$ with sum of coordinates equal to
$1$ whose set of non-zero coordinates corresponds to a simplex in $\Sigma.$

A map $f:P\to Q$ between two polyhedra is called piecewise linear (PL) if it is continuous and the graph $\Gamma_f=\{(x, f(x))| x\in P\}$ is a polyhedron.
A PL-homeomorphism is a PL-map which has a PL inverse (see, e.g., \cite{Hu}).
It is known that any two PL-homeomorphic simplicial complexes can be connected by a finite number of certain combinatorial moves, known as stellar subdivisions and stellar welds  (see, e.g.,
\cite{Lickorish}). These moves have a geometric description as equivariant blowups and blowdowns of the corresponding toric varieties, see \cite{BKF}.

Somewhat unexpectedly, we observe in Theorem \ref{main3}  that $\bD^b(\Coh[\Pp_\Sigma/\bT_{\Pp_\Sigma}])$  is unchanged under blowups of $\Pp_\Sigma.$ This implies that the category depends only on the piecewise linear (PL) homeomorphism type of the  simplicial complex $\Sigma.$ 

\smallskip
\noindent
{\bf Theorem \ref{main4}.}\emph{
Let $\Sigma_1$ and $\Sigma_2$ be two simplicial complexes.
If polyhedra $|\Sigma_1|$ and $|\Sigma_2|$ are PL-homeomorphic, then the  derived categories
$\bD^b(\Coh[U_{\Sigma_1}/\bT_1])$ and $\bD^b(\Coh[U_{\Sigma_2}/\bT_2])$
 of torus-equivariant coherent sheaves on $U_{\Sigma_1}$ and $U_{\Sigma_2}$ respectively are equivalent.
}

It remains unclear whether $\bD^b(\Coh[\Pp_\Sigma/\bT])$ is invariant under arbitrary homeomorphisms of $|\Sigma|$ as opposed to only PL-homeomorphisms.

\smallskip
Finally, in Section \ref{sec.5} we collect two miscellaneous results: one about derived categories of $\bT$\!--equivariant coherent sheaves on noncommutative toric varieties and another about full exceptional collections of equivariant line bundles on $\Pp^2$ without three torus-fixed points.

%\smallskip
{\bf Acknowledgments.} L.B. has been partially supported by NSF grant DMS-1601907. D.O. has been partially supported by Program of the Presidium of the Russian Academy of Sciences No. 01 ``Fundamental Mathematics
and its Applications" under grant PRAS-18-01. We thank Alexander Kuznetsov for multiple useful suggestions. We also thank the Institute for Advanced Study for its hospitality.

\section{Exceptional collections of \texorpdfstring{$\Gm^n$\!}{Gmn}\!--equivariant sheaves on \texorpdfstring{$\Aa^n$}{An}}
\label{sec2}

Let $\Aa^n$ be the n-dimensional affine space over a field $\kk$ and let $\bT=\Gm^n$ be the n-dimensional torus that naturally acts on $\Aa^n.$
Consider the  toric  quotient stack $[\Aa^n/\bT]$ and the abelian category of coherent sheaves $\Coh[\Aa^n/\bT]$ on it.
In this section we describe the bounded derived category of coherent sheaves $\bD^b(\Coh[\Aa^n/\bT])$ on this  toric  quotient stack $[\Aa^n/\bT].$
We will exhibit two full exceptional collections in the derived category
$\bD^b(\Coh[\Aa^n/\bT]),$ one that consists of equivariant line bundles on the whole space and another that consists of certain line bundles on its coordinate subspaces.

%\smallskip
The abelian category of coherent sheaves $\Coh[\Aa^n/\bT]$ on the quotient stack is equivalent to the abelian category of equivariant coherent
sheaves $\Coh^{\bT}\Aa^n$
on the affine space  $\Aa^n$ with respect to the natural action of the torus $\bT=\Gm^n.$ On the other hand,
the abelian category of equivariant coherent sheaves $\Coh^{\bT}\!\Aa^n$ is equivalent to the abelian category of finitely generated
$\Zz^n$\!--graded
modules $\gr^{\Zz^n}\mkern-6mu\operatorname{--}\Sn$ over the polynomial coordinate ring $\Sn=\kk\left[z_1,\ldots,z_n\right]$ with
the natural $\Zz^n$\!--grading. (Specifically, we can consider the
group
$\Zz^n$ of maps from $[n]=\{1,\ldots, n\}$ to $\Zz$ and assign to a monomial
$\prod_{i=1}^n z_i^{a_i}$ the element ${\aa }:[n]\to \Zz$ with ${\aa}(i)=a_i.$)
The abelian category of quasi-coherent sheaves $\Qcoh[\Aa^n/\bT]$ and $\Qcoh^{\bT}\!\Aa^n$ are equivalent to each other and are equivalent
to the abelian category of all
$\Zz^n$\!--graded
modules $\GrZ \Sn.$

%\smallskip
For any grading ${\pp}\in \Zz^n$ we define the equivariant line bundle $\Oo(-{\pp})$ on the affine space $\Aa^n$ as the coherent sheaf that
corresponds to
the module $\Sn(-{\pp})$ which is the free module $\Sn$ with the grading shifted so that the generator
of the module has degree ${\pp}.$ More generally, the twist functor $({\pp})$ on the category $\GrZ \Sn$ is defined as
follows: it takes a graded module $M=\mathop{\bigoplus}_{\qq\in \Zz^n} M_{\qq}$ to
the module $M(\pp)$ for which $M(\pp)_{\qq}= M_{\pp+\qq}$ and takes a morphism
$f: M\to N$ to the same morphism viewed as a morphism between the
twisted modules $f(\pp): M(\pp)\to N(\pp).$

%\smallskip
The following statement is quite immediate.
\begin{proposition}\label{linebundles}
The set of $\bT$\!--equivariant line bundles $\big\{\Oo({\pp}),\;{\pp}\in \Zz^n\big\}$ forms a full exceptional collection in the derived
category $\bD^b(\Coh[\Aa^n/\bT]).$
\end{proposition}

\begin{proof}
Since $\Aa^n$ is affine and $\bT$ is reductive, 
$$
\Ext^*_{[\Aa^n/\bT]}(F,G) = (\Ext^*_{\Aa^n}(F,G))^{\bT}.
$$
Thus there are no non-trivial higher Ext spaces between these sheaves and we have
$$
\Hom_{[\Aa^n/\bT]}(\Oo({\pp}_1),\Oo({\pp}_2)) \cong \left\{
\begin{array}{ll}
\kk, &\text{if }\quad \pp_1\leq \pp_2\\
0, &\text{otherwise.}
\end{array}
\right.
$$
Here the notation $\pp_1\leq \pp_2$ means that $\pp_1(i)\leq \pp_2(i)$ for all $i\in[n].$
We can order $\Oo({\pp})$ lexicographically to ensure that there are no maps in one of the directions.

\smallskip
To prove that the collection is full, observe that every $\bT$\!--equivariant coherent sheaf on $\Aa^n$ corresponds to
a $\Zz^n$\!--graded  module over $\Sn$ which has a finite free $\Zz^n$\!--graded resolution.
\end{proof}

\begin{remark}
{\rm
The collection of Proposition \ref{linebundles}  happens to be a strong exceptional collection.
}
\end{remark}

We will now construct another full exceptional collection in the category $\bD^b(\Coh[\Aa^n/\bT]).$
To describe this collection we introduce some additional notation.

\begin{definition}\label{sheaves}
Let $I\subseteq[n]$ be a subset of $[n]:=\{1,\ldots, n\}.$ Let $\pp:[n]\to\Zz$ be a function.
We denote by ${\FF}_{I, \pp}$ the $\bT$\!--equivariant sheaf on  $\Aa^n$ associated to
the $\Zz^n$\!--graded module
$$
M_{I, \pp}:=\kk\left[z_1,\ldots, z_n\right]/\la z_i,i\in I\ra (-\pp).
$$
\end{definition}

\begin{definition}\label{basic_weights}
We  define  the basic weights
$\ee_i:[n]\to \Zz$ that correspond to the variables $z_i,$ i.e.
there are equalities
$\ee_i(j) = \delta_{ij}.$
For any subset $S\subseteq [n]$
we also introduce the characteristic function  $\cchi_S= \sum_{k\in S} \ee_k.$
\end{definition}

\begin{proposition}\label{ext}
Let ${\FF}_{I, \pp}$ and ${\FF}_{J, \qq}$ be the sheaves in
the category $\Coh[\Aa^n/\bT]$ defined above. Then
$
\Ext^s_{[\Aa^n/\bT]}({\FF}_{I, \pp},  {\FF}_{J, \qq})=0
$
unless the  following two conditions hold:
\begin{itemize}
\item[a)] $(\qq - \pp)\Big|_{\overline{I\cup J}}\le 0,$ i.e. $\qq(i)\leq \pp(i)$ for all $i\not\in I\cup J$
\item[b)] $(\qq - \pp)\Big|_{I\cup J}=\cchi_{S}$ for some subset $S\subseteq [n]$ with $|S|=s$ that satisfies $S\subseteq I\subseteq S\cup J.$
\end{itemize}
%In this case
%$\Ext^s_{[\Aa^n/\bT]}({\FF}_{I, \pp},  {\FF}_{J, \qq})$ is isomorphic to  $\kk$
%for $s=|S|=\sum_{i\in I\cup J}(\qq(i)- \pp(i)).$
\end{proposition}
\begin{proof}
The calculation of Ext spaces between the sheaves ${\FF}_{I, \pp}$ and   ${\FF}_{J, \qq}$ reduces to the calculation of Ext spaces between
graded modules $M_{I, \pp}$ and $M_{J, \qq}$ in the category of graded modules $\GrZ \Sn,$ i.e. we have
$$
\Ext^s_{[\Aa^n/\bT]}({\FF}_{I, \pp},  {\FF}_{J, \qq})\cong\Ext^s_{\GrZ \Sn}(M_{I, \pp}, M_{J, \qq}).
$$

These groups can be computed by taking  projective resolutions of the arguments.
Any module $M_{I,\pp}$ is resolved by the total Koszul complex of the form
\begin{equation}\label{K}
\kK_{I}(-\pp):= \Big\{ \Sn(-\pp-\cchi_{I})\to \cdots \to \bigoplus_{S\subseteq I,| S| =s}
\Sn(-\pp-\cchi_S)\to \cdots \to \Sn(-\pp) \Big\},
\end{equation}
where the part with $|S| = s$ lies in cohomological degree $-s.$ The complex $\kK_{I}$ can be represented as
the tensor product
\[
\kK_{I}\cong\bigotimes_{i\in I}\kK_{\{i\}}:=\bigotimes_{i\in I}\Big\{\Sn(-\ee_i) \lto \Sn \Big\}.
\]
For the dual complex $\kKv_{I}$ we have isomorphisms
\[
\kKv_{I}\cong\bigotimes_{i\in I}\kKv_{\{i\}}\cong\bigotimes_{i\in I}\Big\{\Sn \lto \Sn(\ee_i) \Big\}[-1]\cong
\bigotimes_{i\in I}\kK_{\{i\}}(\ee_i)[-1]\cong\kK_{I}(\cchi_{I})[-| I|],
\]
so the usual non-graded $\Ext$ spaces between $\Sn$\!--modules $M_{I, \pp}$ and $M_{J, \qq}$ are the cohomology of the following tensor product:
\begin{equation*}
%\label{Kprod}
\Ext^s_{\Sn}(M_{I, \pp}, M_{J, \qq})\cong H^s(\kKv_{I}(\pp) \otimes \kK_{J}(-\qq))\cong
H^s(\kK_{I}\otimes\kK_{J}(\pp-\qq+\cchi_{I})[-| I|]).
\end{equation*}

We can rearrange the tensor product as follows
\[
\kK_{I}\otimes\kK_{J}\cong\bigotimes_{i\in I}\kK_{\{i\}}\otimes\bigotimes_{j\in J}\kK_{\{j\}}\cong
\bigotimes_{l\in I\cup J}\kK_{\{l\}}\otimes\bigotimes_{r\in I\cap J}\kK_{\{r\}}\cong
\kK_{I\cup J}\otimes\kK_{I\cap J},
\]
i.e., as the Koszul complex for $\{z_l, l\in I\cup J\}$ tensored with
the Koszul complex for $\{z_r, r\in I\cap J\}.$
Since the complex $\kK_{I\cup J}$
 is a resolution of the module $M_{I\cup J}=\Sn/\la z_l, l\in I\cup J\ra$
we obtain isomorphisms
\[
\kK_{I}\otimes\kK_{J}\cong \kK_{I\cup J}\otimes\kK_{I\cap J}
= \Big\{
M_{I\cup J}(-\cchi_{I\cap J})\to \cdots \to \bigoplus_{\substack{T\subseteq I\cap J,\\ |T| =t}}
M_{I\cup J}(-\cchi_T)\to \cdots \to M_{I\cup J}
\Big\}.
\]
Finally, we  observe that all differentials in this complex are identically zero, because multiplication with any $z_r,$
when $r \in I\cap J,$ is zero in
$M_{I\cup J}=\Sn/\la z_l, l\in I\cup J\ra.$ Therefore, we obtain an isomorphism of complexes
\[
\kK_{I}\otimes\kK_{J}\cong
\bigoplus_{\emptyset\subseteq T\subseteq I\cap J} M_{I\cup J}(-\cchi_{T})[|T|].
\]

\smallskip
As a result, in order to calculate the equivariant $\Ext$ spaces
$
\Ext^*_{[\Aa^n/\bT]}(\FF_{I,\pp},\FF_{J,\qq})
$
we need to find the $\Zz^n$\!--degree zero component of the complex
\begin{equation}\label{ext25}
\kKv_{I}(\pp) \otimes \kK_{J}(-\qq)\cong\kK_{I}\otimes\kK_{J}(\pp-\qq+\cchi_{I})[-| I|]\cong
\bigoplus_{\emptyset\subseteq T\subseteq I\cap J} M_{I\cup J}(\pp-\qq+\cchi_{I}-\cchi_{T})[-| I|+|T|].
\end{equation}
We have a nontrivial (one-dimensional) $\Zz^n$\!--degree zero component in $M_{I\cup J}(\pp-\qq+\cchi_{I}-\cchi_{T})[-| I|+|T|]$
if and only if
$\pp-\qq+\cchi_{I}-\cchi_{T}$ takes nonnegative values for all $i\not\in I\cup J$ and takes zero values
for all $i\in I\cup J.$
This implies that $(\qq-\pp)|_{\overline{I\cup J}}\le 0$ and $(\qq-\pp)|_{I\cup J}$
determines $T\subseteq I\cap J$ uniquely.
Thus, the difference $(\qq-\pp)|_{I\cup J}$ should be the characteristic function $\cchi_S$ of some subset $S\subseteq I$ that contains
$I\backslash (I\cap J).$
In this case $T$ consists of all $r \in I\cap J$ for which $\qq(r)-\pp(r)=0,$ i.e. the subset $T\in I\cap J$ is the complement to $S$
in $I.$
Observe that the cohomological degree is then equal to
\[
| I\backslash T | =| S|= \sum_{i\in I\cup J} (\qq(i)-\pp(i)).
\]
Thus, we obtain the following description of Ext spaces:
\begin{align}\label{descrip_Ext}
\Ext^r_{[\Aa^n/\bT]}(\FF_{I,\pp},\FF_{J,\qq})\cong \bigoplus_{I\backslash (I\cap J) \subseteq S\subseteq I}
H^{r-|S|} ([\Aa^n/\bT], \FF_{I\cup J}(\pp-\qq+\cchi_{S})).
\end{align}
This implies the statement
\[
\Ext^s_{[\Aa^n/\bT]}({\FF}_{I, \pp},  {\FF}_{J, \qq}) \cong
\left\{
\begin{array}{ll}
\kk, &\text{if}\quad (\qq - \pp)\Big|_{\overline{I\cup J}}\le 0,\quad
\text{and}\quad (\qq - \pp)\Big|_{I\cup J}=\cchi_{S}\; \\
&
\text{ for some } S \text{ with } |S|=s \text{ and } S\subseteq I\subseteq S\cup J
,\\
0, &\text{otherwise.}
\end{array}
\right.
\]
\end{proof}

\begin{remark}\label{rem25}
The above proof also contains a calculation of (the modules corresponding to) the $\Ext$ sheaves between $\FF_{I,\pp}$ and $\FF_{J,\qq}$
in \eqref{ext25} which will be used in the next section.
\end{remark}
We will now observe that a subset of  these sheaves forms a full exceptional collection.

\begin{proposition}\label{cn}
The sheaves ${\FF}_{I, \pp},$ where $I$ runs over all subsets of $[n]$ and $\pp:[n]\to\Zz$ runs over all functions with the property $\pp(i)=0$
for all $i\not\in I,$ form a full exceptional collection in the derived category $\bD^b(\Coh[\Aa^n/\bT])$
under the lexicographic ordering by $\pp$ and $|I|.$
\end{proposition}
\begin{proof}
Since $\pp_2(i)=\pp_1(i)=0$ for all $i\not\in I_1\cup I_2,$ by Proposition \ref{ext} the Ext spaces 
$$
\Ext^s_{[\Aa^n/\bT]}( {\FF}_{I, \pp_1}, {\FF}_{I, \pp_2})
$$
are zero unless $\pp_2 -\pp_1=\cchi_{S}\geq 0,$ and in the case $\pp_2=\pp_1$ we have $S=\emptyset,$ and $I_1\subseteq I_2.$ This ensures that we get an exceptional collection.

To show that the exceptional collection is full, we run induction on strata. By picking $I=[n]$ we see that
all twists $ {\FF}_{[n], \pp}$ of the structure sheaf of the point ${\zz}=0$ are in the exceptional collection, since there are no
restrictions on $\pp.$ Suppose that we have proved that the subcategory  of $\bD^b(\Coh[\Aa^n/\bT])$
generated by the exceptional collection
contains  ${\FF}_{I', \pp}$  for all $\pp$ and all proper supersets $I'\supset I.$ Then we can use the short
exact sequences for each $k\not\in I$
\[
0\lto \FF_{I,\pp+\ee_k} \lto \FF_{I, \pp} \lto \FF_{I\cup\{k\}, \pp}\to 0
\]
to go from $\FF_{I, \pp}$ with $\pp(i)=0$ for all $i\not\in I$ to $\FF_{I, \pp}$ with arbitrary $\pp.$
Indeed, by induction assumption, the last term of the above short exact sequence is in the subcategory, which allows us to
add or subtract $\ee_k$ to any $\pp$ with $\FF_{I, \pp}$ in the subcategory and still stay within it.
This eventually allows us to consider $\FF_{\emptyset, \pp},$ which are the equivariant sheaves $\OO(-\pp).$
Thus the exceptional collection generates all of $\bD^b(\Coh[\Aa^n/\bT])$  by Proposition \ref{linebundles}.
\end{proof}

\begin{example}
{\rm
For $n=1$ the exceptional collection is given by
$$
\langle
\ldots, \FF_{\{1\},-2}, \FF_{\{1\},-1} ,  \FF_{\emptyset ,0}, \FF_{\{1\},0},  \FF_{\{1\},1}, \FF_{\{1\},2}, \ldots
\rangle
$$
where we identify $\pp:[1]\to \Zz$ with $\pp(1).$ The nontrivial (one-dimensional) morphism spaces are
$
\Hom( \FF_{\emptyset ,0} ,\FF_{\{1\},0}) 
$ 
and $
\Ext^1( \FF_{\{1\} ,k} ,\FF_{\{1\},k+1}) 
$
for all $k\in \Zz.$ One can also verify that the composition of morphisms is nonzero in the only possible case
of 
$$
\Ext^1( \FF_{\{1\} ,-1} ,  \FF_{\emptyset ,0}) \times \Hom(   \FF_{\emptyset ,0},\FF_{\{1\},0}) 
\to \Ext^1( \FF_{\{1\} ,-1},\FF_{\{1\},0}).
$$
}
\end{example}

\begin{remark}
{\rm
Note that the shifted sheaves $\big\{{\FF}_{I, \pp}[\bp]\big\},$ where $\bp=\sum_{i=1}^n \pp(i),$ form a strong full exceptional
collection
in the derived category $\bD^b(\Coh[\Aa^n/\bT]).$
}
\end{remark}

\bigskip
\begin{remark}
{\rm
The calculation of Proposition \ref{cn} is a particular case of a more general phenomenon applicable to the equivariant derived categories of
toric varieties and stacks which we will consider in the next section.
}
\end{remark}

\begin{remark}
{\rm
We will illustrate this exceptional collection in the case $n=2$ in diagram  \eqref{diagram1}.
The square $\Box$ corresponds to the line bundle $\FF=\FF_{\emptyset,{\bf 0}}$ on the whole $\Aa^2.$ The bullets $\bullet$ indicate the line bundles supported on the coordinate lines and $\circ$ indicate the equivariant rank one skyscraper sheaves at the origin in $\Aa^2.$ The locations of the $\circ$ points are in accordance with the grading, whereas $\bullet$ and $\Box$ are offset to the left and/or down.
The solid arrows indicate nonzero $\Hom$ spaces which correspond to the restriction maps to smaller coordinate subspaces in $\Aa^2,$ i.e. $S=\emptyset.$
The dashed arrows indicate nonzero $\Ext^1$ spaces and the dotted arrows stand for $\Ext^2$ spaces.
}
\end{remark}

%\newpage

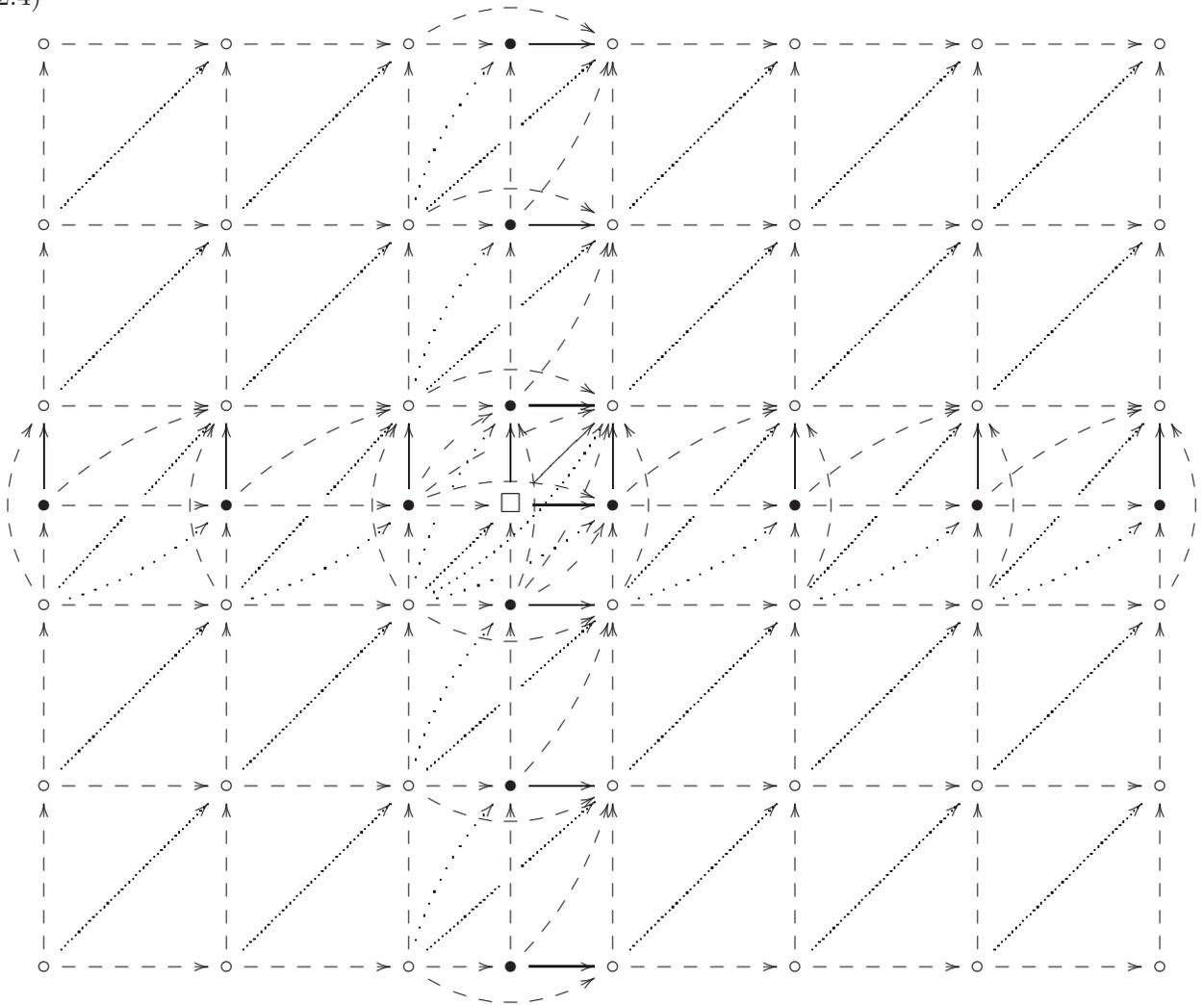
\begin{align}\label{diagram1}
\xymatrix{
\circ\ar@[blue]@{-->}[rr] &&
\circ\ar@[blue]@{-->}[rr] & &\circ \ar@[blue]@{-->}[r]\ar@/^3ex/@[blue]@{-->}[rr] &\bullet\ar@[red][r] &
\circ\ar@[blue]@{-->}[rr]&&\circ\ar@[blue]@{-->}[rr]&&
\circ\ar@[blue]@{-->}[rr]&&\circ
\\
\\
\circ\ar@[blue]@{-->}[rr]\ar@[blue]@{-->}[uu]\ar@[red]@{.>}[uurr] &&
\circ\ar@[blue]@{-->}[rr]\ar@[blue]@{-->}[uu]\ar@[red]@{.>}[uurr] &&
\circ \ar@[blue]@{-->}[r]\ar@[blue]@{-->}[uu]\ar@[red]@{.>}[uurr] |!{[r];[uur]}\hole
\ar@/^/@[red]@{.>}[uur] \ar@/^3ex/@[blue]@{-->}[rr]
&
\bullet\ar@[red][r]\ar@[blue]@{-->}[uu]\ar@/_/@[blue]@{-->}[uur] &
\circ\ar@[blue]@{-->}[rr]\ar@[blue]@{-->}[uu]\ar@[red]@{.>}[uurr]&&\circ\ar@[blue]@{-->}[rr]\ar@[blue]@{-->}[uu]\ar@[red]@{.>}[uurr]&&
\circ\ar@[blue]@{-->}[rr]\ar@[blue]@{-->}[uu]\ar@[red]@{.>}[uurr]&&\circ\ar@[blue]@{-->}[uu]
\\
\\
\circ\ar@[blue]@{-->}[rr]\ar@[blue]@{-->}[uu]\ar@[red]@{.>}[uurr] &&
\circ\ar@[blue]@{-->}[rr]\ar@[blue]@{-->}[uu]\ar@[red]@{.>}[uurr] & &\circ
\ar@[blue]@{-->}[r]\ar@[blue]@{-->}[uu]\ar@[red]@{.>}[uurr]|!{[r];[uur]}\hole
\ar@/^/@[red]@{.>}[uur]\ar@/^3ex/@[blue]@{-->}[rr] &\bullet\ar@[red][r]\ar@[blue]@{-->}[uu]\ar@/_/@[blue]@{-->}[uur] &
\circ\ar@[blue]@{-->}[rr]\ar@[blue]@{-->}[uu]\ar@[red]@{.>}[uurr]&&\circ\ar@[blue]@{-->}[rr]\ar@[blue]@{-->}[uu]\ar@[red]@{.>}[uurr]&&
\circ\ar@[blue]@{-->}[rr]\ar@[blue]@{-->}[uu]\ar@[red]@{.>}[uurr]&&
\circ\ar@[blue]@{-->}[uu]
\\
\bullet\ar@[blue]@{-->}[rr]\ar@[red][u]\ar@[blue]@{-->}@/^/[urr] &&
\bullet\ar@[blue]@{-->}[rr]\ar@[red][u]\ar@[blue]@{-->}@/^/[urr] &&
\bullet \ar@[blue]@{-->}[r]\ar@[red][u]\ar@/^/@[blue]@{-->}[urr]|!{[r];[ur]}\hole \ar@/^/@[blue]@{-->}[ur]\ar@/^2ex/@[blue]@{-->}[rr]
&
\Box\ar@[red][r]\ar@[red][u]\ar@[red][ur] &
\bullet\ar@[blue]@{-->}[rr]\ar@[red][u]\ar@/^/@[blue]@{-->}[urr]&&\bullet\ar@[blue]@{-->}[rr]\ar@[red][u]\ar@/^/@[blue]@{-->}[urr]&&
\bullet\ar@[blue]@{-->}[rr]\ar@[red][u]\ar@/^/@[blue]@{-->}[urr]&&\bullet\ar@[red][u]
\\
\circ\ar@[blue]@{-->}[rr]\ar@[blue]@{-->}[u]\ar@[red]@{.>}[uurr]|!{[u];[urr]}\hole \ar@[red]@{.>}@/_/[urr]\ar@/^3ex/@[blue]@{-->}[uu]
&&
\circ\ar@[blue]@{-->}[rr]\ar@[blue]@{-->}[u]\ar@[red]@{.>}[uurr]|!{[u];[urr]}\hole \ar@[red]@{.>}@/_/[urr]\ar@/^3ex/@[blue]@{-->}[uu]
&&
\circ \ar@[blue]@{-->}[r]\ar@[blue]@{-->}[u]\ar@/_2ex/@[red]@{.>}[uurr]\ar@[red]@{.>}[ur] \ar@/^/@[red]@{.>}[uur]|!{[u];[ur]}\hole
\ar@/_/@[red]@{.>}[urr]|!{[r];[ur]}\hole
\ar@/_3ex/@[blue]@{-->}[rr]\ar@/^3ex/@[blue]@{-->}[uu]
&
\bullet\ar@[red][r]\ar@[blue]@{-->}[u]\ar@/_/@[blue]@{-->}[uur]|!{[u];[ur]}\hole \ar@/_/@[blue]@{-->}[ur]
\ar@/_2ex/@[blue]@{-->}[uu]
&
\circ\ar@[blue]@{-->}[rr]\ar@[blue]@{-->}[u]\ar@[red]@{.>}[uurr]|!{[u];[urr]}\hole
\ar@/_/@[red]@{.>}[urr]\ar@/_3ex/@[blue]@{-->}[uu]&&
\circ\ar@[blue]@{-->}[rr]\ar@[blue]@{-->}[u]\ar@[red]@{.>}[uurr]|!{[u];[urr]}\hole
\ar@/_/@[red]@{.>}[urr]\ar@/_3ex/@[blue]@{-->}[uu]&&
\circ\ar@[blue]@{-->}[rr]\ar@[blue]@{-->}[u]\ar@[red]@{.>}[uurr]|!{[u];[urr]}\hole
\ar@/_/@[red]@{.>}[urr]\ar@/_3ex/@[blue]@{-->}[uu]
&&
\circ\ar@[blue]@{-->}[u]\ar@/_3ex/@[blue]@{-->}[uu]
\\
\\
\circ\ar@[blue]@{-->}[rr]\ar@[blue]@{-->}[uu]\ar@[red]@{.>}[uurr] &&
\circ\ar@[blue]@{-->}[rr]\ar@[blue]@{-->}[uu]\ar@[red]@{.>}[uurr] & &\circ
\ar@[blue]@{-->}[r]\ar@[blue]@{-->}[uu]\ar@[red]@{.>}[uurr]|!{[r];[uur]}\hole
\ar@/^/@[red]@{.>}[uur]\ar@/_3ex/@[blue]@{-->}[rr] &\bullet\ar@[red][r]\ar@[blue]@{-->}[uu]\ar@/_/@[blue]@{-->}[uur] &
\circ\ar@[blue]@{-->}[rr]\ar@[blue]@{-->}[uu]\ar@[red]@{.>}[uurr]&&\circ\ar@[blue]@{-->}[rr]\ar@[blue]@{-->}[uu]\ar@[red]@{.>}[uurr]&&
\circ\ar@[blue]@{-->}[rr]\ar@[blue]@{-->}[uu]\ar@[red]@{.>}[uurr]&&\circ\ar@[blue]@{-->}[uu]
\\
\\
\circ\ar@[blue]@{-->}[rr]\ar@[blue]@{-->}[uu]\ar@[red]@{.>}[uurr] &&
\circ\ar@[blue]@{-->}[rr]\ar@[blue]@{-->}[uu]\ar@[red]@{.>}[uurr] & &\circ
\ar@[blue]@{-->}[r]\ar@[blue]@{-->}[uu]\ar@[red]@{.>}[uurr]|!{[r];[uur]}\hole
\ar@/^/@[red]@{.>}[uur]\ar@/_3ex/@[blue]@{-->}[rr] &\bullet\ar@[red][r]\ar@[blue]@{-->}[uu]\ar@/_/@[blue]@{-->}[uur] &
\circ\ar@[blue]@{-->}[rr]\ar@[blue]@{-->}[uu]\ar@[red]@{.>}[uurr]&&\circ\ar@[blue]@{-->}[rr]\ar@[blue]@{-->}[uu]\ar@[red]@{.>}[uurr]&&
\circ\ar@[blue]@{-->}[rr]\ar@[blue]@{-->}[uu]\ar@[red]@{.>}[uurr]&&\circ\ar@[blue]@{-->}[uu]
}
\end{align}
%\begin{align*}
%{\rm Figure ~1.~ Exceptional~collection~on~}\Aa^2/\bT.
%\end{align*}

Let us now describe some basic indecomposable morphisms between sheaves of the form $\FF_{I, \pp}.$
At first, let us fix a weight $\pp:[n]\to\Zz$ and consider a pair of sheaves $\FF_{I, \pp}$ and $\FF_{J, \pp}$
such that $J=I\cup \{j\}$ for some $j\not\in I.$
In this case there is a canonical morphism
\[
\rho_{I, \pp}^{j}: \FF_{I, \pp}\lto \FF_{J, \pp},
\]
which is actually the natural restriction of the sheaf $\FF_{I, \pp}$ on the affine subvariety
$\Aa^{n-|J|}=\left\{ z_k=0 \, |\;\text{for all}\; k\in J \right\}.$ This morphism corresponds
to $S=\emptyset$ in Proposition \ref{ext}.
It is easy to see that the restriction map from $\FF_{I, \pp}$ to $\FF_{J, \pp}$
for any $J\supset I$
is equal to a composition of the maps of the form $\rho_{J_k, \pp}^{j_k}$ for some $J_k$ and $j_k\not\in J_k.$
In particular, in the case $J=I\cup \{j_1, j_2\}$ we have equality
\begin{equation}
\rho_{J_1, \pp}^{j_2}\circ \rho_{I, \pp}^{j_1}=\rho_{J_2, \pp}^{j_1}\circ \rho_{I, \pp}^{j_2},
\end{equation}
where $J_r=I\cup \{j_r\}$ with $r=1,2.$

Second, let us consider the case when $I_1= I\backslash\{i\}$ for some element $i\in I.$
In this situation we have a canonical element
\[
\psi_{I, \pp}^{i}\in \Ext^1(\FF_{I, \pp}, \FF_{I_1, \pp_1})\cong\kk,
\]
when $\pp_1=\pp+\ee_i.$ This element is uniquely defined as a morphism
from $\FF_{I, \pp}$ to $\FF_{I_1, \pp_1}[1]$ that determines the following
short exact sequence
\begin{align}\label{short}
\xymatrix{
0\ar[r]& \FF_{I_1, \pp_1} \ar^{\cdot z_{i}}[r]& \FF_{I_1, \pp} \ar^{\rho_{I_1, \pp}^i}[r]& \FF_{I, \pp}\ar[r]& 0.
}
\end{align}

The following lemmas say that all nontrivial $\Ext$\!--spaces in our exceptional collection are generated by morphisms of the form $\rho_{I, \pp}^{j}$ of degree $0$ and morphisms
$\psi_{I, \pp}^{i}$
of degree $1.$

\begin{lemma}\label{comp_prim}
Let $I, J\subseteq [n]$ be two subsets and let $S\subseteq I$ with $|S|=s$  satisfy $I \subseteq J\cup S.$
Let $i\in I$ be an element and $I_1=I\backslash \{i\}$  be the complement.
Then for any $\pp$ and $\qq=\pp+\cchi_{S}$ we have
\begin{itemize}
\item[1)] if $i\not\in S,$ then the composition
\[
\Ext^{s}(\FF_{I,\pp},\FF_{J,\qq})\otimes \Hom(\FF_{I_{1},\pp},\FF_{I,\pp})\stackrel{\sim}{\lto} \Ext^s(\FF_{I_1,\pp},\FF_{J,\qq})
\]
is an isomorphism;
\item[2)] if $i\in S,$ then the composition
\[
\Ext^{s-1}(\FF_{I_1,\pp_1},\FF_{J,\qq})\otimes \Ext^1(\FF_{I,\pp},\FF_{I_1,\pp_1})\stackrel{\sim}{\lto} \Ext^s(\FF_{I,\pp},\FF_{J,\qq})
\]
is an isomorphism, where $\pp_1=\pp+\ee_i.$
\end{itemize}
\end{lemma}
\begin{proof}
Let us consider the standard short exact sequence \eqref{short}
where $\pp_1=\pp+\ee_{i}$ and $\FF_{I_1,\pp_1}=\FF_{I_1, \pp}(-\ee_{i_1}).$
Recall that this extension between $\FF_{I, \pp}$ and $\FF_{I_1, \pp_1}$ is given by
the element $\psi_{I, \pp}^{i}\in\Ext^1(\FF_{I,\pp},\FF_{I_1,\pp_1}).$

If $i\not\in S,$ then $\qq(i)-\pp_1(i)=-1$ and, hence, by Proposition \ref{ext} we have vanishing
\[
\Ext^t(\FF_{I_1, \pp_1}, \FF_{J, \qq})=0
\]
for all $t.$ The long exact sequence obtained by applying $\Ext(-,\FF_{J,\qq})$ to \eqref{short},  implies that the natural map
\[
\Ext^{s}(\FF_{I,\pp},\FF_{J,\qq})\stackrel{\sim}{\lto} \Ext^s(\FF_{I_1,\pp},\FF_{J,\qq})
\]
is an isomorphism. But this map is exactly the composition with the nontrivial element
$\rho_{I_1, \pp}^i$ that spans $\Hom(\FF_{I_1, \pp}, \FF_{I, \pp}).$

If $i\in S,$ then $\qq(i)-\pp(i)=1$ and $i\not\in I_1.$ By Proposition \ref{ext} we have vanishing
\[
\Ext^t(\FF_{I_1, \pp}, \FF_{J, \qq})=0
\]
for all $t.$ The same long exact sequence as above  implies that the natural map
\[
\Ext^{s-1}(\FF_{I_1, \pp_1}, \FF_{J, \qq})\stackrel{\sim}{\lto}
\Ext^s(\FF_{I, \pp}, \FF_{J, \qq})
\]
is an isomorphism. But this map is exactly the composition with the generator
$\psi_{I, \pp}^{i}$ of $\Ext^{1}(\FF_{I, \pp}, \FF_{I_1, \pp_1}).$
\end{proof}

Let $I, J$ be two arbitrary subsets of $[n],$ and let $S$ be a subset of  $I$ that contains $I\backslash (I\cap J)$ so that 
for any $\pp$ and $\qq=\pp+\cchi_{S}$ by Proposition \ref{ext} we have $\Ext^s(\FF_{I,\pp},\FF_{I',\pp'})\cong \kk,$ when $s=|S|.$
Consider an arbitrary ordering $(i_1,\ldots, i_s)$  of all elements of the set $S.$ It defines a sequence of subsets $I=I_0\supset I_1\supset\ldots\supset I_s,$ where $I_k\subseteq I$ is the complement to the subset $\{i_1,\ldots, i_k\}\subseteq S\subseteq I$ in $I.$
For any $0\le k< s$ we know that
$\Ext^1(\FF_{I_{k},\pp_{k}},\FF_{I_{k+1},\pp_{k+1}})\cong \kk,$ where $\pp_0=\pp$ and $\pp_k$ is defined by induction $\pp_k=\pp_{k-1}+\ee_k.$
Moreover, there is the canonical generator $\psi_{I_k, \pp_{k}}^{i_{k+1}}$ of this space that defines the extension.
The set $I_s$ is contained in $J$ and $\qq=\pp_s.$
Therefore, there is the restriction morphism from $\FF_{I_s, \pp_s}$ to $\FF_{J, \qq}$ that corresponds to the empty subset of $I_s.$
Now we can prove the following lemma.

\begin{lemma}\label{decomp}
Let $I, J\subseteq [n]$ be two sets and let a subset $S\subseteq I$ contain $I\backslash (I\cap J).$
Let $i_1,\ldots, i_s$ be an arbitrary ordering of all elements of the set $S$ and
$I_k\subseteq I$ be the complement to the subset $\{i_1,\ldots, i_k\}\subseteq S\subseteq I$ in $I$ for any $0\le k\le s.$
Then for any $\pp$ and $\qq=\pp+\cchi_{S}$ the composition map
\begin{align}\label{decomp_form}
\Hom(\FF_{I_s,\pp_s},\FF_{J,\qq})\otimes \Ext^1(\FF_{I_{s-1},\pp_{s-1}},\FF_{I_s,\pp_s})\otimes\cdots\otimes \Ext^1(\FF_{I,\pp},\FF_{I_1,\pp_1})\stackrel{\sim}{\lto} \Ext^s(\FF_{I,\pp},\FF_{J,\qq})
\end{align}
is an isomorphism, where $\pp_k=\pp+\sum_{l=1}^k \ee_l.$
\end{lemma}
\begin{proof}
We prove this lemma using induction on $s=|S|.$ For $S=\emptyset$ there is nothing to prove.
Consider $I_1\subset I$ that is the complement to $\{i_1\}.$ By induction hypothesis the composition map
\[
\Hom(\FF_{I_s,\pp_s},\FF_{J,\qq})\otimes \Ext^1(\FF_{I_{s-1},\pp_{s-1}},\FF_{I_s,\pp_s})\otimes\cdots\otimes \Ext^1(\FF_{I_1,\pp_1},\FF_{I_2,\pp_2})\stackrel{\sim}{\lto} \Ext^{s-1}(\FF_{I_1,\pp_1},\FF_{J,\qq})
\]
is an isomorphism. Now, it is sufficient to show that the composition
\[
\Ext^{s-1}(\FF_{I_1,\pp_1},\FF_{J,\qq})\otimes\Ext^1(\FF_{I,\pp},\FF_{I_1,\pp_1})\stackrel{\sim}{\lto} \Ext^s(\FF_{I,\pp},\FF_{J,\qq})
\]
is an isomorphism, which follows from Lemma \ref{comp_prim}.
\end{proof}

We can explicitly determine the combinatorial condition that implies that the composition of the $\Ext$ spaces is nonzero. Specifically,
suppose we have nonzero $\Ext^s(\FF_{I,\pp},\FF_{J,\qq})\ne 0$ and $\Ext^t(\FF_{J,\qq},\FF_{K,\rr})\ne 0,$
which correspond to sets $S$ and $T$  of cardinalities $s=|S|$ and $t=|T|,$ with $S\subseteq I\subseteq S\cup J$ and $T\subseteq J\subseteq T\cup K.$

By Proposition \ref{ext} we have $\qq-\pp=\cchi_{S}$ and $\rr-\qq=\cchi_{T}.$
Thus, if  $S\cap T\neq\emptyset,$ then the difference $\rr-\pp$ is not a characteristic function, because it takes value
$2$ on $S\cap T.$ In this case, the composition is clearly zero. Otherwise, to
 ensure $\Ext^{s+t}(\FF_{I,\pp},\FF_{K,\rr})\neq 0,$ we must have
\[
S\sqcup T \subseteq I \subseteq S\sqcup T \cup K.
\]
The second inclusion is automatically satisfied. Indeed, since $S\cup J\supseteq I$ and $T\cup K\supseteq J$
we have $(S\sqcup T)\cup K\supseteq I.$ Therefore,  we just need to make sure
that $T\subseteq I.$ Thus, $T$ should belong to $I\cap J$ and $S\cap T=\emptyset.$
We will show below that under these necessary conditions  the natural composition is indeed nonzero.

\begin{proposition}\label{compos}
Let $I, J, K\subseteq [n]$ be three sets and let subsets $S\subseteq I$ and $T\subseteq J$
contain $I\backslash (I\cap J)$ and $J\backslash(J\cap K),$ respectively.
Let $\FF_{I, \pp}, \FF_{J, \qq}, \FF_{K, \rr} $ be the sheaves in
the category $\Coh[\Aa^n/\bT]$ such that $\qq=\pp+\cchi_{S}$ and $\rr=\qq+\cchi_{T}.$
Assume that $T\subseteq I\cap J$ and $S\cap T=\emptyset.$ Then the following properties hold
\begin{itemize}
\item[1)] the spaces
$\Ext^{s}(\FF_{I,\pp},\FF_{J,\qq}), \Ext^t(\FF_{J,\qq},\FF_{K,\rr}),$ and  $\Ext^{s+t}(\FF_{I,\pp},\FF_{K,\rr})$
are not trivial, when $s=|S|$ and $t=|T|.$
\item[2)] the natural composition
\[
\Ext^{t}(\FF_{J,\qq},\FF_{K,\rr})\otimes \Ext^s(\FF_{I,\pp},\FF_{J,\qq})\stackrel{\sim}{\lto} \Ext^{s+t}(\FF_{I,\pp},\FF_{K,\rr})
\]
is an isomorphism.
\end{itemize}
\end{proposition}
\begin{proof}
The spaces
$\Ext^{s}(\FF_{I,\pp},\FF_{J,\qq}), \Ext^t(\FF_{J,\qq},\FF_{K,\rr})$
are not trivial for $s=|S|$ and $t=|T|$ by Proposition \ref{ext}.
Under assumption $T\subseteq I\cap J$ and $S\cap T=\emptyset$ the set $W:=S\sqcup T$ is the subset of $I$
that contains $I\backslash (I\cap K).$
Since $\rr-\pp=\cchi_{W},$ again by Proposition \ref{ext} the space $\Ext^{s+t}(\FF_{I,\pp},\FF_{K,\rr})$
is not trivial.

By Lemma \ref{decomp} the space $\Ext^{s}(\FF_{I,\pp},\FF_{J,\qq})$ can be decomposed as in (\ref{decomp_form})
\[
\Hom(\FF_{I_s,\pp_s},\FF_{J,\qq})\otimes \Ext^1(\FF_{I_{s-1},\pp_{s-1}},\FF_{I_s,\pp_s})\otimes\cdots\otimes \Ext^1(\FF_{I,\pp},\FF_{I_1,\pp_1})\stackrel{\sim}{\lto} \Ext^s(\FF_{I,\pp},\FF_{J,\qq}).
\]
The space $\Hom(\FF_{I_s,\pp_s},\FF_{J,\qq})$ can also be decomposed as
\[
\Hom(\FF_{I_s,\pp_s},\FF_{J,\qq})\cong \Hom(\FF_{J_{t-1},\qq}, \FF_{J,\qq})\otimes\cdots\otimes \Hom(\FF_{I_s,\pp_s},\FF_{J_{1},\qq}),
\]
where $I_s=J_0\subset J_1\subset\ldots\subset J_{t-1}\subset J_t=J$ is a sequence of subsets such that $|J_l\backslash J_{l-1}|=1$ for all $0< l\le t.$
Now we can apply Lemma \ref{comp_prim} part 1) to argue that the composition
\[
\Ext^{t}(\FF_{J_l,\qq},\FF_{K,\rr})\otimes \Hom(\FF_{J_{l-1},\qq},\FF_{J_l,\qq})\stackrel{\sim}{\lto} \Ext^{t}(\FF_{J_{l-1},\qq},\FF_{K,\rr})
\]
is an isomorphism for all $0< l\le t.$ And by the same Lemma \ref{comp_prim} part 2) we have isomorphisms
\[
\Ext^{t+s-k}(\FF_{I_k,\pp_k},\FF_{K,\rr})\otimes \Ext^1(\FF_{I_{k-1},\pp_{k-1}},\FF_{I_k,\pp_k})\stackrel{\sim}{\lto} \Ext^{t+s-k+1}(\FF_{I_{k-1},\pp_{k-1}},\FF_{K,\rr})
\]
for all $0<k\le s.$ This completes the proof of  the proposition.
\end{proof}

Finally, we describe relations between morphisms.
Let now $J=I\backslash\{i_1, i_2\},$ where $i_1, i_2\in I.$
We put $I_r=I\backslash\{i_r\}$ and $\pp_r=\pp+\ee_{i_r}$ for $r=1,2$ and define $\qq=\pp+\ee_{i_1}+\ee_{i_2}.$
Consider the following diagram of  morphisms in the derived category

\begin{align*}
\xymatrix{
\FF_{J, \qq} \ar^{\cdot z_{i_1}}[r]\ar_{\cdot z_{i_2}}[d] & \FF_{J, \pp_2} \ar^{\rho_{J, \pp_2}^{i_1}}[r]\ar_{\cdot z_{i_2}}[d] & \FF_{I_2, \pp_2}\ar^{\psi_{I_2, \pp_2}^{i_1}}[r]\ar_{\cdot z_{i_2}}[d]
&\FF_{J, \qq}[1]\ar^{\cdot z_{i_2}[1]}[d] \\
\FF_{J, \pp_1} \ar^{\cdot z_{i_1}}[r]\ar_{\rho_{J, \pp_1}^{i_2}}[d] & \FF_{J, \pp} \ar^{\rho_{J, \pp}^{i_1}}[r]\ar_{\rho_{J, \pp}^{i_2}}[d] & \FF_{I_2, \pp}\ar^{\psi_{I_2, \pp}^{i_1}}[r]\ar_{\rho_{I_2, \pp}^{i_2}}[d]
&\FF_{J, \pp_1}[1]\ar^{\rho_{J, \pp_1}^{i_2}[1]}[d] \\
\FF_{I_1, \pp_1} \ar^{\cdot z_{i_1}}[r]\ar_{\psi_{I_1, \pp_1}^{i_2}}[d]& \FF_{I_1, \pp} \ar^{\rho_{I_1, \pp}^{i_1}}[r]\ar_{\psi_{I_1, \pp}^{i_2}}[d]& \FF_{I, \pp}\ar^{\psi_{I, \pp}^{i_1}}[r] \ar_{\psi_{I, \pp}^{i_2}}[d]\ar@{}[dr]|{-} &\FF_{I_1, \pp_1}[1]\ar^{\psi_{I_1, \pp_1}^{i_2}[1]}[d]\\
\FF_{J, \qq}[1] \ar_{\cdot z_{i_1}[1]}[r] & \FF_{J, \pp_2}[1] \ar_{\rho_{J, \pp}^{i_1}[1]}[r] & \FF_{I_2, \pp_2}[1]\ar_{\psi_{I_2, \pp_2}^{i_1}[1]}[r]
&\FF_{J, \qq}[2]
}
\end{align*}

\bigskip
\noindent
in which all rows and columns are exact triangles. It follows from the properties of a triangulated category
that all squares in this diagram commute except the one marked $"-"$ which anticommutes (see, e.g., \cite{BBD}, Proposition 1.1.11).
This implies the following equality of morphisms:

\begin{equation}\label{relat}
\begin{array}{ll}
1)& \psi_{I, \pp}^{i_1}\circ \rho_{I_2, \pp}^{i_2}=\rho_{J, \pp_1}^{i_2}[1]\circ\psi_{I_2, \pp}^{i_1},
\qquad
\psi_{I, \pp}^{i_2}\circ \rho_{I_1, \pp}^{i_1}=\rho_{J, \pp_2}^{i_1}[1]\circ\psi_{I_1, \pp}^{i_2},
\\
2)& \psi_{I_1, \pp_1}^{i_2}[1]\circ\psi_{I, \pp}^{i_1}= -\psi_{I_2, \pp_2}^{i_1}[1]\circ  \psi_{I, \pp}^{i_2}.
\end{array}
\end{equation}

\medskip
Consider now a nontrivial space $\Ext^{s}(\FF_{I,\pp},\FF_{J,\qq})\cong \kk,$ where  $\qq=\pp+\cchi_{S}$ for a subset $S\subseteq I$
containing $I\backslash (I\cap J).$
Relations (\ref{relat}) and Lemma \ref{decomp} imply that there is  a fixed (up to sign) element of the space $\Ext^{s}(\FF_{I,\pp},\FF_{J,\qq})\cong \kk$
that is a composition of morphisms $\rho_{I, \pp}^{j}$ of degree $0$ and elements
of the form  $\psi_{I, \pp}^{i}$ of degree $1.$
Moreover, we can choose signs such that all morphisms will commute.

\begin{definition}\label{phi} Let $I_1=I\cup \{i\}.$ We  define
$\phi_{I, \pp}^{i}\in \Ext^1(\FF_{I, \pp}, \FF_{I_1, \pp_1})$ as $(-1)^{\varepsilon}\psi_{I, \pp}^{i},$ where $\varepsilon=\mathop{\sum}\limits_{t=1}^{i-1} \pp(t).$
\end{definition}

Now the second line of (\ref{relat}) turns to
\begin{align}\label{realcom}
\phi_{I_1, \pp_1}^{i_2}[1]\circ\phi_{I, \pp}^{i_1}= \phi_{I_2, \pp_2}^{i_1}[1]\circ  \phi_{I, \pp}^{i_2}.
\end{align}

For any function $\pp: [n]\to\Zz$ define the support $\supp\pp$ as a subset of $[n]$ consisting of all $i,$
for which $\pp(i)\ne 0.$ Summarizing the above results we obtain the following description of 
the full exceptional collection.

\begin{theorem}\label{main1}
The derived category $\bD^b(\Coh[\Aa^n/\bT])$ possesses a strong full exceptional collection $\big\{{\FF}_{I, \pp}[\bp]\big\},$ which can be described as follows.
\begin{itemize}
\item[1)] Objects are  shifted sheaves of the form ${\FF}_{I, \pp}[\bp],$ where  $\pp:[n]\to\Zz$ is any function, the subset $I\subseteq[n]$ contains $\supp\pp,$ and
$\bp=\sum_{i=1}^n \pp(i).$ The ordering is lexicographic in $\pp$ and $|I|;$ 
\item[2)] The space of morphisms $\Hom_{[\Aa^n/\bT]}({\FF}_{I, \pp}[\bp],  {\FF}_{J, \qq}[\bq])$ is nontrivial only if $(\qq - \pp)=\cchi_{S}$
for some subset $S\subseteq I$ that contains $I\backslash (I\cap J).$  In this case the space is isomorphic to $\kk,$ and there is
a fixed nontrivial element
\[
\phi_{I, \pp; S, J}\in \Hom_{[\Aa^n/\bT]}({\FF}_{I, \pp}[\bp],  {\FF}_{J, \qq}[\bq]).
\]
\item[3)] The composition $\phi_{J, \qq; T, K}\circ\phi_{I, \pp; S, J}$ is nontrivial only if $\qq=\pp+\cchi_{S},$ $T\subseteq I\cap J$ and $S\cap T=\emptyset.$
In this case, there is an equality
\[
\phi_{J, \qq; T, K}\circ\phi_{I, \pp; S, J}=\phi_{I, \pp; S\sqcup T, K}.
\]
\end{itemize}
\end{theorem}

\begin{proof}
It was proved in Proposition \ref{cn} that the set of sheaves $\big\{{\FF}_{I, \pp}[\bp]\big\},$ where $\pp:[n]\to\Zz$ is any function and the subset $I$ contains $\supp\pp,$ forms an
exceptional collection. Proposition \ref{ext} implies that this collection is strong when sheaves ${\FF}_{I, \pp}$ are shifted with $\bp=\sum_{i=1}^n \pp(i).$
From the same Proposition \ref{ext} we know that the space $\Hom_{[\Aa^n/\bT]}({\FF}_{I, \pp}[\bp],\;  {\FF}_{J, \qq}[\bq])$ is nontrivial only when $(\qq - \pp)=\cchi_{S}$
for some subset $S\subseteq I$ containing $I\backslash (I\cap J).$ In this case the space is one-dimensional.

Using  notations and a decomposition formula from Lemma \ref{decomp} we can define $\phi_{I, \pp; S, J}$ as a composition of elements
\[
\phi_{I_{k-1}, \pp_1}^{i_{k}}\in
\Ext^1(\FF_{I_{k-1},\pp_{k-1}},\; \FF_{I_k,\pp_k})
\]
 for $0<k\le s$ and the restriction map
$\rho_{I_s, \qq; J}\in \Hom(\FF_{I_s,\qq}, \;\FF_{J,\qq}).$ The definition of $\phi_{I, \pp; S, J}$ does not depend on the ordering of elements $(i_1,\ldots, i_s)$ of $S$
due to commutation relations (\ref{realcom}). Moreover, by Lemma \ref{decomp} all such elements are nontrivial.

Finally, by Proposition \ref{compos} the composition $\phi_{J, \qq; T, K}\circ\phi_{I, \pp; S, J}$ is nontrivial when $\qq=\pp+\cchi_{S},$ $T\subseteq I\cap J$ and $S\cap T=\emptyset.$
Moreover, commutation relations 1) of (\ref{relat}) and (\ref{realcom}) imply that  $\phi_{J, \qq; T, K}\circ\phi_{I, \pp; S, J}=\phi_{I, \pp; S\sqcup T, K}$ whenever
this composition is defined and nontrivial.
\end{proof}

\begin{remark}
{\rm We illustrate the case $n=2$ (without the cohomological shifts) in the diagram
\eqref{diagramprimitive}.
Specifically, we write explicitly the set $I$ and the grading $\pp=(\pp(1),\pp(2))$ for the elements $\FF_{I,\pp}$ of the full exceptional collection of Proposition \ref{cn} and only write the indecomposable morphisms. As in the diagram \eqref{diagram1}, we indicate $\Hom$ spaces by a solid arrow and $\Ext^1$ spaces by a dashed arrow.
}
\end{remark}

\newpage

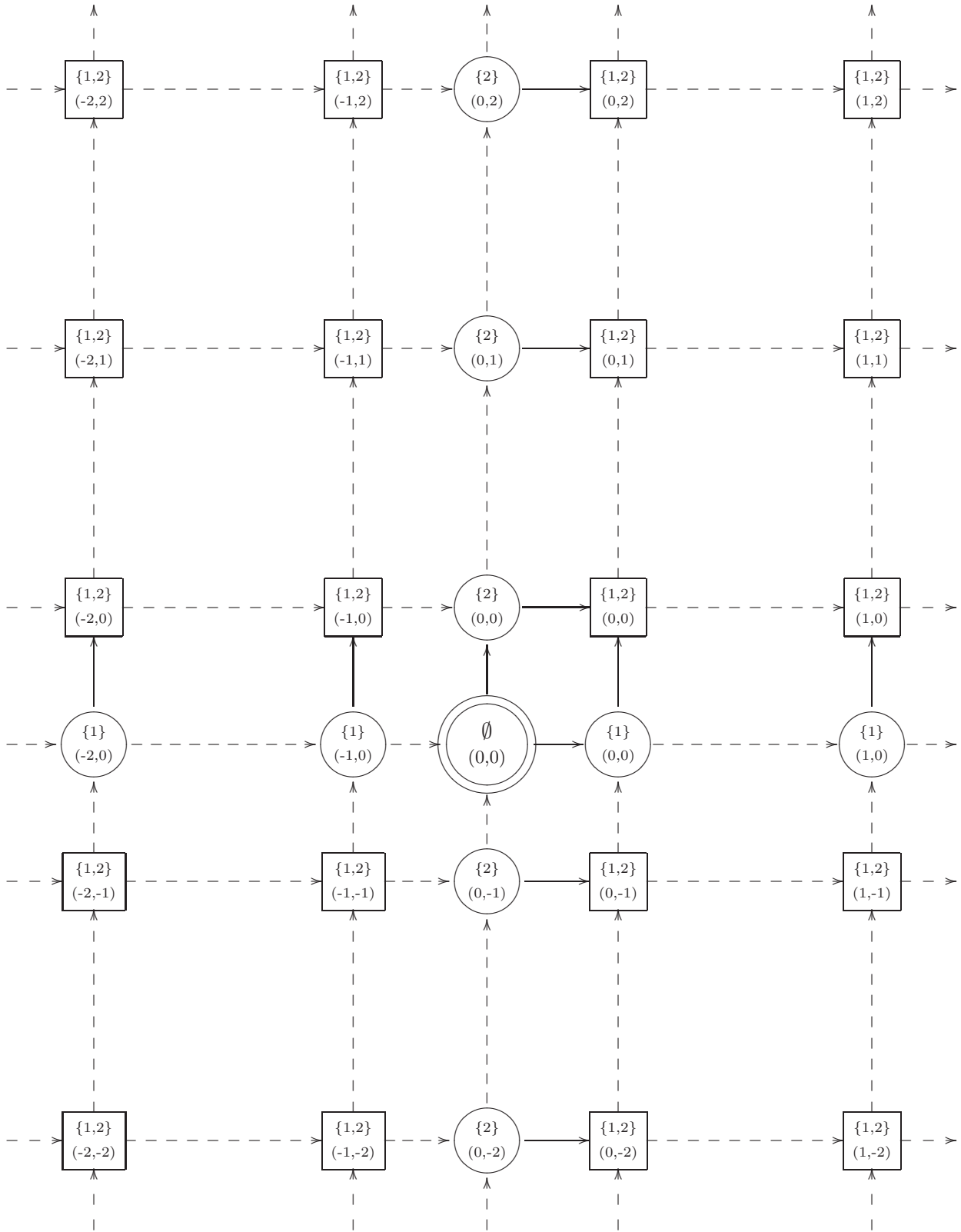
\begin{align}\label{diagramprimitive}
\xymatrix{
&\textcolor{blue}{}
&&&
\textcolor{blue}{}
&
\textcolor{blue}{}
&
\textcolor{blue}{}
&&&
\textcolor{blue}{}
&
\\
\textcolor{blue}{}
\ar@[blue]@{-->}[r]
&
*+[F]\txt{\tiny \{1,2\}\\ \tiny (-2,2)}\ar@[blue]@{-->}[rrr]\ar@[blue]@{-->}[u]
&&&
*+[F]\txt{\tiny \{1,2\}\\ \tiny (-1,2)}\ar@[blue]@{-->}[r]\ar@[blue]@{-->}[u]
&
*++[o][F]\txt{\tiny \{2\}\\ \tiny (0,2)}\ar@[red][r] \ar@[blue]@{-->}[u]
&
*+[F]\txt{\tiny \{1,2\}\\ \tiny (0,2)}\ar@[blue]@{-->}[rrr] \ar@[blue]@{-->}[u]
&&&
*+[F]\txt{\tiny \{1,2\}\\ \tiny (1,2)}\ar@[blue]@{-->}[r] \ar@[blue]@{-->}[u]
&
\textcolor{blue}{}
\\
\\
\\
\textcolor{blue}{}
\ar@[blue]@{-->}[r]
&
*+[F]\txt{\tiny \{1,2\}\\ \tiny (-2,1)}\ar@[blue]@{-->}[rrr]\ar@[blue]@{-->}[uuu]
&&&
*+[F]\txt{\tiny \{1,2\}\\ \tiny (-1,1)} \ar@[blue]@{-->}[r]
\ar@[blue]@{-->}[uuu]
&
*++[o][F]\txt{\tiny \{2\}\\ \tiny (0,1)}\ar@[red][r]
\ar@[blue]@{-->}[uuu]
&
*+[F]\txt{\tiny \{1,2\}\\ \tiny (0,1)}\ar@[blue]@{-->}[rrr]
\ar@[blue]@{-->}[uuu]
&&&
*+[F]\txt{\tiny \{1,2\}\\ \tiny (1,1)}\ar@[blue]@{-->}[r]
\ar@[blue]@{-->}[uuu]
&
\textcolor{blue}{}
\\
\\
\\
\textcolor{blue}{}
\ar@[blue]@{-->}[r]
&
*+[F]\txt{\tiny \{1,2\}\\ \tiny (-2,0)}\ar@[blue]@{-->}[rrr]\ar@[blue]@{-->}[uuu]
&&&
*+[F]\txt{\tiny \{1,2\}\\ \tiny (-1,0)} \ar@[blue]@{-->}[r]\ar@[blue]@{-->}[uuu]
&
*++[o][F]\txt{\tiny \{2\}\\ \tiny (0,0)}\ar@[red][r]\ar@[blue]@{-->}[uuu]
&
*+[F]\txt{\tiny \{1,2\}\\ \tiny (0,0)}\ar@[blue]@{-->}[rrr]\ar@[blue]@{-->}[uuu]
&&&
*+[F]\txt{\tiny \{1,2\}\\ \tiny (1,0)}\ar@[blue]@{-->}[r]\ar@[blue]@{-->}[uuu]
&
\textcolor{blue}{}
\\
\textcolor{blue}{}
\ar@[blue]@{-->}[r]
&
*++[o][F]\txt{\tiny \{1\}\\ \tiny (-2,0)}\ar@[blue]@{-->}[rrr]\ar@[red][u]
&&&
*++[o][F]\txt{\tiny \{1\}\\ \tiny (-1,0)} \ar@[blue]@{-->}[r]\ar@[red][u]
&
*+++[o][F=]\txt{$\emptyset$\\ \footnotesize  (0,0)}\ar@[red][r]\ar@[red][u]
&
*++[o][F]\txt{\tiny \{1\}\\ \tiny (0,0)}\ar@[blue]@{-->}[rrr]\ar@[red][u]
&&&
*++[o][F]\txt{\tiny \{1\}\\ \tiny (1,0)}\ar@[blue]@{-->}[r]\ar@[red][u]
&
\textcolor{blue}{}
\\
\textcolor{blue}{}
\ar@[blue]@{-->}[r]
&
*+[F]\txt{\tiny \{1,2\}\\ \tiny (-2,-1)}\ar@[blue]@{-->}[rrr]\ar@[blue]@{-->}[u]
&&&
*+[F]\txt{\tiny \{1,2\}\\ \tiny (-1,-1)} \ar@[blue]@{-->}[r]\ar@[blue]@{-->}[u]
&
*++[o][F]\txt{\tiny \{2\}\\ \tiny (0,-1)} \ar@[red][r]\ar@[blue]@{-->}[u]
&
*+[F]\txt{\tiny \{1,2\}\\ \tiny (0,-1)} \ar@[blue]@{-->}[rrr]\ar@[blue]@{-->}[u]
&&&
*+[F]\txt{\tiny \{1,2\}\\ \tiny (1,-1)}\ar@[blue]@{-->}[r]\ar@[blue]@{-->}[u]
&
\textcolor{blue}{}
\\
\\
\\
\textcolor{blue}{}
\ar@[blue]@{-->}[r]
&
*+[F]\txt{\tiny \{1,2\}\\ \tiny (-2,-2)}\ar@[blue]@{-->}[rrr]\ar@[blue]@{-->}[uuu]
&&&
*+[F]\txt{\tiny \{1,2\}\\ \tiny (-1,-2)} \ar@[blue]@{-->}[r]\ar@[blue]@{-->}[uuu]
&
*++[o][F]\txt{\tiny \{2\}\\ \tiny (0,-2)}\ar@[red][r]\ar@[blue]@{-->}[uuu]
&
*+[F]\txt{\tiny \{1,2\}\\ \tiny (0,-2)} \ar@[blue]@{-->}[rrr]\ar@[blue]@{-->}[uuu]
&&&
*+[F]\txt{\tiny \{1,2\}\\ \tiny (1,-2)} \ar@[blue]@{-->}[r]\ar@[blue]@{-->}[uuu]
&
\textcolor{blue}{}
\\
&\textcolor{blue}{}\ar@[blue]@{-->}[u]
&&&
\textcolor{blue}{}\ar@[blue]@{-->}[u]
&
\textcolor{blue}{}\ar@[blue]@{-->}[u]
&
\textcolor{blue}{}\ar@[blue]@{-->}[u]
&&&
\textcolor{blue}{}\ar@[blue]@{-->}[u]
&
}
\end{align}
%\begin{align*}
%{\rm Figure~2.~Indecomposable~morphisms~for~}\Aa^2/\bT.
%\end{align*}

Finally, we can reformulate the results of Theorem \ref{main1} by a certain reindexing of the elements of the collection. 
\begin{theorem}\label{main1.5}
The derived category $\bD^b(\Coh[\Aa^n/\bT])$ of $\bT$\!--equivariant coherent sheaves on $\Aa^n$ possesses a strong full exceptional collection of objects $E_{a}$ indexed by $a=(a_1,\ldots,a_n)\in\Zz^n$
ordered lexicographically.
The morphism spaces are described as follows.
\[
\Hom(E_a,E_b)=
\left\{
\begin{array}{ll}
\kk,&{\rm if~ for~all~}i {\rm~either~} b_i-a_i\in \{0,1\} {\rm ~or~} (b_i,a_i) = (1,-1),\\
0,&{\rm otherwise.}
\end{array}
\right.
\]
The nontrivial spaces $\Hom(E_a,E_b)$ come equipped with a distinguished generator $\rho_{a,b}.$
The composition of morphisms $\Hom(E_a,E_b)\otimes \Hom(E_b,E_c)\to \Hom(E_a,E_c)$ is nontrivial if and only if all three spaces are nontrivial. In these cases $\rho_{b,c}\circ\rho_{a,b}=\rho_{a,c}.$
\end{theorem}

\begin{proof}
This  is a reformulation of Theorem \ref{main1} obtained by reindexing
\[
\FF_{I,\pp}[\bp] =: E_{a}
\]
for $a=(a_1,\ldots,a_n)$ with
\begin{align}\label{rules}
a_i =\left\{
\begin{array}{ll}
\pp(i),&{\rm if~} \pp(i)< 0 {\rm~or~} i\not\in I,\\
\pp(i)+1,&{\rm if~}\pp(i)\geq 0 {\rm~and~} i\in I.
\end{array}
\right.
\end{align}
In the opposite direction, given $a=(a_1,\ldots,a_n)$ we can recover 
$I$ and $\pp$ by 
\begin{align}\label{rulesback}
I=\{i,a_i\neq 0\};
~~~~~
\pp(i) = \left\{
\begin{array}{ll}
a_i,&{\rm if~}a_i\leq 0,\\
a_i-1,&{\rm if~} a_i> 0 .
\end{array}
\right.
\end{align}
This bijection relies crucially on the property $\pp(i)=0$ for all $i\not\in I$ of Theorem \ref{main1}.

We will now show that the non-triviality condition of Theorem \ref{main1} translates into the 
condition of the present theorem.

For $(I,\pp)$ and $(J,\qq)$ from Theorem \ref{main1} with nontrivial  $\Hom_{[\Aa^n/\bT]}({\FF}_{I, \pp}[\bp],  {\FF}_{J, \qq}[\bq])$ we have 
$\qq-\pp=\chi_S$ with $S\subseteq I \subseteq S\cup J.$ 
Consider $a$ and $b$ in $\Zz^n$ that correspond to $(I,\pp)$ and $(J,\qq)$ respectively.
For $i\not\in S$ we have $\pp(i)=\qq(i).$ Moreover, either $i$ lies in both $I$ and $J,$ or in none of $I$ and $J,$ or only in $J.$ In the first two cases $a_i=b_i$ since the rules \eqref{rules} apply the same way to $\pp(i)$ and $\qq(i).$ If $i$ lies only in $J,$ then it is possible that 
the second line of \eqref{rules} applies to $\qq(i)$ while the the first line applies to $\pp(i)$ which leads to $b_i=a_i+1,$ else $b_i=a_i.$
For $i\in S$ we have $\qq(i)=\pp(i)+1.$ The only possible way that the second line would apply to $\qq(i)$ while the first line would apply to $\pp(i)$ is when $\pp(i)=-1$ and $\qq(i)=0;$ this leads to $(b_i,a_i)=(1,-1).$ Otherwise, we have $b_i-a_i\in\{0,1\}.$

In the other direction, suppose $a$ and $b$ satisfy the property that for all $i$ either 
$b_i-a_i\in \{0,1\}$ or $(b_i,a_i)=(1,-1).$  Consider the corresponding $(I,\pp)$ and $(J,\pp)$ obtained by the rules \eqref{rulesback}.
For a given $i$ there are several possibilities and we will build the set $S$ with $\qq-\pp=\chi_S$ accordingly. If $a_i=b_i\neq 0,$ then $i\in I\cap J$ and $\pp(i)=\qq(i).$ If $a_i=b_i=0,$ then $i\not\in I\cup J$ and $\pp(i)=\qq(i)=0.$ If $a_i=b_i-1$ and $a_i\neq 0,$ then the same lines of \eqref{rulesback} apply to both $a$ and $b.$ 
We get $\qq(i)-\pp(i)=1$ and include this $i$ in $S.$ If $a_i=0$ and $b_i=1,$ then we get $i\in J,$ $i\not\in I,$ and $\qq(i)-\pp(i)=0-0=0.$ In this case $i\not\in S.$ Finally, if $a_i=-1$ and $b_i=1,$ we get $i\in I\cap J$ and $\qq(i)-\pp(i)=0-(-1)=1$ and $i\in S.$ By construction, $S\subseteq I$ since $i\in I$ for all cases where $i\in S.$ Similarly, if $i\in I,$ then  $i\in J$ unless $b_i=0.$ This means $a_i=-1$ which leads to $i\in S.$ Thus $I\subseteq S\cup J$ and we see that there is a nontrivial $\Hom$ space between $\FF_{I,\pp}[\bp]$ and $\FF_{J,\qq}[\bq].$

The statements about generators and composition of morphisms now follow immediately from Theorem \ref{main1}. The (non-unique) ordering is derived from the description of the $\Hom(E_a,E_b)$ spaces.
\end{proof}

\section{
Exceptional collection of
\texorpdfstring{$\Gm^n$\!}{Gmn}\!--equivariant
sheaves on
\texorpdfstring{$\Aa^n$}{An}
without coordinate subspaces
}

\label{sec3}

In this section we will construct an analog of the exceptional collection of Theorem \ref{cn} for an arbitrary nonempty
open subset $U_\Sigma$ of $\Aa^n$ which is
invariant with respect to the torus $\bT=\Gm^n.$

\smallskip
Specifically,
let $\Sigma$ be a simplicial complex on the set $[n]:=\{1,\ldots, n\},$ i.e. an inclusion closed nonempty collection of subsets of this set.
We associate to $\Sigma$ an open subset $U_\Sigma\subseteq \Aa^n$ as follows. We view a (closed)  point
 $(z_1,\ldots,z_n)\in \Aa^n$ as a map ${\zz}:[n]\to \kk$ with $\zz(i)=z_i.$ Then
$$
U_\Sigma:= \{{\zz}, \mbox{ such that }{\zz}^{-1}(0)\in \Sigma\}= \bigcup_{\sigma\in \Sigma}U_\sigma,
$$
where $U_\sigma\subset \Aa^n$ is the open subset of all $(z_1,\ldots,z_n)$ such that $z_i\neq 0$ for all $i\not\in \sigma.$
Equivalently, a point with coordinates $(z_1,\ldots,z_n)$ lies in $U_\Sigma$ if and only if there exists $\sigma\in\Sigma$
such that $z_i\neq 0$ for all $i\not\in\sigma.$ It is easy to see that all open $\bT$\!--invariant subsets of $\Aa^n$
can be uniquely written in the form $U_\Sigma.$

As in Definition \ref{basic_weights}, we  introduce the weights
$\ee_i:[n]\to \Zz$ that correspond to the variables $z_i,$ i.e.
we have
$\ee_i(j) = \delta_{ij}.$
For any subset $S\subseteq [n]$
we also introduce the characteristic function  $\cchi_S= \sum_{k\in S} \ee_k.$

For each subset $I\subseteq [n]$ and a grading $\pp:[n]\to \Zz$ we defined  $\bT$\!--equivariant coherent sheaves  ${\FF}_{I, \pp}$ on
$\Aa^n$ which are associated with
the $\Zz^{n}$\!--graded modules
\begin{equation}\label{fip}
M_{I, \pp}:=\kk\left[z_1,\ldots, z_n\right]/\la z_i,i\in I\ra (-\pp).
\end{equation}
over the polynomial ring $\kk[z_1,\ldots,z_n]$ (see Definition \ref{sheaves}).
The natural open embedding $\js: U_{\Sigma}\hookrightarrow \Aa^n$ gives an inverse image functor
$\js^{*}:\Coh^{\bT}\!\Aa^n \to \Coh^{\bT}\!U_{\Sigma}$ that is exact. The functor $\js^*$ induces
the functor
between derived categories of equivariant coherent sheaves for which we will use the same notation.

\begin{definition}
For each $I\subseteq [n]$ and $\pp:[n]\to \Zz$ define the $\bT$\!--equivariant coherent sheaf ${\FF}^\Sigma_{I, \pp}$ on $U_\Sigma$ as
the restriction of the $\bT$\!--equivariant sheaf
${\FF}_{I, \pp}$ on $\Aa^n$ associated to \eqref{fip} to the open subset $U_\Sigma,$ i.e. ${\FF}^\Sigma_{I, \pp}= \js^*{\FF}_{I, \pp}.$
\end{definition}

\begin{remark}\label{rem3.2}
{\rm
If $I\not\in\Sigma,$ then ${\FF}_{I, \pp}$ is supported on the complement of $U_\Sigma,$ so  ${\FF}^\Sigma_{I, \pp}=0.$
}
\end{remark}

At first, we calculate cohomology of sheaves  ${\FF}^\Sigma_{I, \pp}$ for special weights $\pp.$ More precisely, we have the following vanishing
conditions.

\begin{lemma}\label{cohom}
Let $I\subseteq [n]$ and let $\pp:[n]\to \Zz$ be a function such that
$\pp(i)=0$ for all $i\not\in I.$ Then $H^s([U_\Sigma/\bT], {\FF}^\Sigma_{I, \pp})=0$ for all $s$
unless  the following conditions hold:
$I\in\Sigma,$ $\pp=0$ and $s=0.$
In this case we have
$H^0([U_\Sigma/\bT], {\FF}^\Sigma_{I, \00})\cong \kk.$
Moreover, the natural restriction map 
$$H^0([\Aa^n/\bT], {\FF}^\Sigma_{I, \00})\to
H^0([U_\Sigma/\bT], {\FF}^\Sigma_{I, \00})
$$
is an isomorphism.
\end{lemma}
\begin{proof}
By Remark \ref{rem3.2} we only need to look at the case $I\in\Sigma.$

To calculate cohomology of the sheaf ${\FF}^\Sigma_{I, \pp}$ we will use  the ${\rm \check Cech}$ complex.
It is built from the
affine open sets $U_\sigma \subseteq U_\Sigma$ for $\sigma\in \Sigma,$ defined by
$
U_{\sigma} = \{ {\zz}\; |\; {\zz}^{-1}(0)\subseteq \sigma \}.
$
Equivalently,  $U_\sigma$ is the set of all $(z_1,\ldots,z_n)$ such that $z_i\neq 0$ for all $i\not\in \sigma.$ We have
$U_\Sigma = \bigcup_{\sigma\in \Sigma}U_\sigma$ and $U_\sigma\cap U_\tau = U_{\sigma \cap \tau}$
which allows us to define a ${\rm \check Cech}$ covering of $U_\Sigma$ with affine subsets.

The first step is to calculate the degree zero part of the space of sections of
$\FF^\Sigma_{I, \pp}$ on $U_\sigma.$ The whole space of sections is given by the module
$$
\kk[z_1,\ldots z_n]/\la z_i,i\in I\ra [z_k^{-1},k\not\in\sigma] (-\pp).
$$
As a consequence, it is only nonzero if $I\subseteq \sigma.$ If $I\subseteq \sigma,$ the degree
zero part is either $0$ or $\kk.$ Specifically, it is one-dimensional  if and only if $\pp(i) = 0$
 for all $i\in I$ and $\pp(i)\leq 0$
when $i\in \sigma\setminus I.$ However, by assumption we have $\pp(i)=0$ for all $ i\not\in I.$
Therefore,  $H^0([U_{\sigma}/\bT], {\FF}^\Sigma_{I, \pp})\ne 0$ only when $\pp= 0$ and $I\subseteq \sigma.$
This implies that
$H^s([U_\Sigma/\bT], {\FF}^\Sigma_{I, \pp})=0$ for all $s$ when $\pp\neq 0.$

Let now $\pp= 0.$
Consider the subset (not a subcomplex) $\Star(I)\subseteq \Sigma$ that consists of all elements $\sigma\in \Sigma$ that contain $I.$ When one
defines  \v{C}ech complex, one picks an ordering of the elements of $\Sigma,$ which in turn induces an ordering on $\Star(I).$ Note that we can ignore the parts of the \v{C}ech complex that do not 
come from intersections of $U_\sigma,\sigma\in \Star(I)$ because they do not contribute. From the
usual definition of \v{Cech} differential, we see that \v{C}ech differential of the cohomology of the Koszul differential can be identified
with
\begin{equation}\label{C}
0\to \bigoplus_{\sigma\in \Star(I)} \kk \to  \bigoplus_{\{\sigma_1,\sigma_2\}\subseteq\Star(I)}
\kk \to \cdots \to  \bigoplus_{\substack{J\subseteq \Star(I), \\ |J|=l\quad}}\kk \to  \bigoplus_{\substack{J\subseteq \Star(I),\\ |J|=l+1}}\kk
  \to\cdots\to  \kk \to 0
\end{equation}
with the components of the differential nonzero only when there is an inclusion of the corresponding subsets of size $l$ and $l+1.$
Given an inclusion, the map is $(\pm 1),$ depending on the parity of the location of the additional element in the induced ordering. This is a
standard complex which has only one nontrivial cohomology at the left term that is isomorphic to $\kk.$
Thus $H^s([U_\Sigma/\bT], {\FF}^\Sigma_{I, \00})=0,$ when $s\ne 0$ and $H^0([U_\Sigma/\bT], {\FF}^\Sigma_{I, \00})\cong \kk.$
\end{proof}

The inverse image functor $\js^*$ induces the natural maps
\[
\Ext_{[\Aa^n/\bT]}^s({\FF}_{I_1, \pp_1}, \FF_{I_2, \pp_2})\lto \Ext_{[U_\Sigma/\bT]}^s({\FF}^\Sigma_{I_1, \pp_1}, \FF^\Sigma_{I_2, \pp_2})
\]
for any sheaves ${\FF}_{I_1, \pp_1}, \FF_{I_2, \pp_2}$ and their restrictions on $U_{\Sigma}.$
The following proposition describes these maps in a special case which is sufficient for our purposes.

\begin{proposition}\label{calcext}
Let $I_1$ and $I_2$ be elements of $\Sigma$ and let $\pp_1, \pp_2$ be functions $[n]\to \Zz$ such that
$\pp_r(i)=0$ for all $i\not\in I_r.$ Then
% there are only two cases
\begin{itemize}
\item[a)] if $I_1\cup I_2\not\in\Sigma,$ then $\Ext_{[U_\Sigma/\bT]}^s({\FF}^\Sigma_{I_1, \pp_1}, \FF^\Sigma_{I_2, \pp_2})=0$ for all $s.$
\item[b)] if $I_1\cup I_2\in\Sigma,$ then the natural map
\[
\Ext_{[\Aa^n/\bT]}^s({\FF}_{I_1, \pp_1}, \FF_{I_2, \pp_2})
%\stackrel{\sim}{\lto}
\lto
\Ext_{[U_\Sigma/\bT]}^s({\FF}^\Sigma_{I_1, \pp_1}, \FF^\Sigma_{I_2, \pp_2})
\]
is an isomorphism.
\end{itemize}
In particular, $\Ext_{[U_\Sigma/\bT]}^s({\FF}^\Sigma_{I_1, \pp_1}, \FF^\Sigma_{I_2, \pp_2})$ is nontrivial (and is isomorphic to $\kk$) iff $I_1\cup I_2\in\Sigma$ and 
$\pp_2=\pp_1+\cchi_{S}$ for some subset $S$ with $s=|S|$ and $S \subseteq I_1\subseteq S\cup I_2.$ 
\end{proposition}
\begin{proof}
In the proof of Proposition \ref{ext} we showed that Koszul resolutions give isomorphisms (\ref{descrip_Ext}):
\begin{align}\label{iso1}
\Ext^r_{[\Aa^n/\bT]}(\FF_{I,\pp},\FF_{J,\qq})\cong \bigoplus_{I\backslash (I\cap J) \subseteq S\subseteq I}
H^{r-|S|} ([\Aa^n/\bT], \FF_{I\cup J}(\pp-\qq+\cchi_{S})).
\end{align}
Unfortunately, the argument of Proposition \ref{ext} is not directly applicable to the current situation,
because $U_\Sigma$ is not affine and the invertible sheaves used to resolve $\FF_{I,\pp}$ and $\FF_{J,\qq}$ are no longer projective objects. Nonetheless, there is a local-to-global spectral sequence
from $H^*([U_{\Sigma}/\bT], {\mathcal Ext}^*(\FF^{\Sigma}_{I,\pp},\FF^{\Sigma}_{J,\qq}))$ to
$\Ext_{[U_{\Sigma}/\bT]}^*(\FF^{\Sigma}_{I,\pp},\FF^{\Sigma}_{J,\qq}).$ The sheaves $\mathcal Ext^*$ can be computed using resolutions by equivariant line bundles as in Proposition \ref{ext}
 (see Remark \ref{rem25})
and the $E_2$ sheet of the spectral sequence consists of 
\[
E_2^{r-s,s} = \bigoplus_{I\backslash (I\cap J) \subseteq S\subseteq I,~|S|=s} H^{r-s} ([U_{\Sigma}/\bT], \FF^{\Sigma}_{I\cup J}(\pp-\qq+\cchi_{S})).
\]
Lemma \ref{cohom} is applicable here since $\pp-\qq+\cchi_{S}$ is zero on the complement of $I\cup J$ and it implies that only $r=s$ terms are nonzero. Therefore, the local-to-global spectral sequence degenerates
and gives isomorphisms
\begin{align}\label{iso2}
\Ext^r_{[U_{\Sigma}/\bT]}(\FF^{\Sigma}_{I,\pp},\FF^{\Sigma}_{J,\qq})\cong \bigoplus_{I\backslash (I\cap J) \subseteq S\subseteq I} H^{r-|S|} ([U_{\Sigma}/\bT], \FF^{\Sigma}_{I\cup J}(\pp-\qq+\cchi_{S})).
\end{align}

Isomorphisms (\ref{iso1}) and (\ref{iso2}) commute with the inverse image functor, 
since the former can also be obtained from the local-to-global spectral sequence which is functorial. 
Another application of Lemma \ref{cohom} now implies the proposition.
\end{proof}

\begin{proposition}\label{exc_col}
The equivariant sheaves $\FF^\Sigma_{I,\pp},$ where $I$ runs over all elements of $\Sigma$ and $\pp:[n]\to\Zz$ runs over all functions with the property $\pp(i)=0$
for all $i\not\in I,$ form a full exceptional collection in the derived category
$\bD^b(\Coh[U_\Sigma/\bT]).$ The collection is ordered lexicographically by $\pp$ and then $|I|.$ 
\end{proposition}

\begin{proof}
Proposition \ref{calcext}
shows that there are no $\Ext$ in the direction opposite to the order.
To prove that the exceptional collection is full, we observe that the functor $j_\Sigma^*$ is essentially surjective and therefore the equivariant line bundles $\FF^\Sigma_{\emptyset,\pp}$ generate $\bD^b(\Coh[U_\Sigma/\bT]).$ 
We now observe that the argument of Proposition \ref{cn} applies word for word
in this new situation. The only difference is that the sheaves $\FF_{I, \pp}$ are zero for $I\not\in \Sigma,$ which does not affect the proof.
\end{proof}
Proposition \ref{exc_col} is a part of more general and precise Theorem \ref{main2} below.

\begin{theorem}\label{main2}
The derived category $\bD^b(\Coh[U_\Sigma/\bT])$ possesses a strong full exceptional collection $\big\{{\FF}^{\Sigma}_{I, \pp}[\bp]\big\},$ which can be described as follows.
\begin{itemize}
\item[1)] Objects are  shifted sheaves of the form ${\FF}^{\Sigma}_{I, \pp}[\bp],$ where  $\pp:[n]\to\Zz$ is any function, the subset $I\in\Sigma$ contains $\supp\pp,$ and
$\bp=\sum_{i=1}^n \pp(i).$
\item[2)] The space of morphisms $\Hom_{[U_\Sigma/\bT]}({\FF}^{\Sigma}_{I, \pp}[\bp],  {\FF}^{\Sigma}_{J, \qq}[\bq])$ is nontrivial only if $(I\cup J)\in \Sigma$ and  $\qq=\pp+\cchi_{S}$
for some subset $S\subseteq I$ that contains $I\backslash (I\cap J).$  In this case the space is isomorphic to $\kk,$ and there is
a distinguished nontrivial element
\[
\phi_{I, \pp; S, J}\in \Hom_{[U_\Sigma/\bT]}({\FF}^{\Sigma}_{I, \pp}[\bp],  {\FF}^{\Sigma}_{J, \qq}[\bq]).
\]
\item[3)] The composition $\phi_{J, \qq; T, K}\circ\phi_{I, \pp; S, J}$ is nontrivial only if $I\cup K\in \Sigma$ and  $\qq=\pp+\cchi_{S},$ $T\subseteq I\cap J,$ $S\cap T=\emptyset.$
In this case, there is an equality
\[\phi_{J, \qq; T, K}\circ\phi_{I, \pp; S, J}=\phi_{I, \pp; S\sqcup T, K}.
\]
\end{itemize}
\end{theorem}

\begin{proof}
Proposition \ref{exc_col} tells us  that the collection $\big\{{\FF}^{\Sigma}_{I, \pp}[\bp]\big\},$ where $\pp:[n]\to\Zz$ is any function and the subset $I\in \Sigma$ contains $\supp\pp,$ forms an
exceptional collection. By Proposition \ref{calcext} this collection is strong when sheaves ${\FF}^{\Sigma}_{I, \pp}$ are shifted with $\bp=\sum_{i=1}^n \pp(i).$
This proposition also implies that $\Hom_{[U_\Sigma/\bT]}({\FF}^{\Sigma}_{I, \pp}[\bp],  {\FF}^{\Sigma}_{J, \qq}[\bq])$ is nontrivial only when
$(I\cup J)\in \Sigma$ and $(\qq - \pp)=\cchi_{S}$
for some subset $S\subseteq I$ containing $I\backslash (I\cap J).$ In this case the space is one-dimensional.

We define $\phi_{I, \pp; S, J}\in \Hom_{[U_\Sigma/\bT]}({\FF}^{\Sigma}_{I, \pp}[\bp],  {\FF}^{\Sigma}_{J, \qq}[\bq])$ as the images
of the similar elements from $\Hom_{[\Aa^n/\bT]}({\FF}_{I, \pp}[\bp],  {\FF}_{J, \qq}[\bq])$ under inverse image functor
$\js^*.$ By Proposition \ref{calcext} they  are nontrivial and property 3) follows from property 3) of Theorem \ref{main1}.
\end{proof}

It is informative to compare the exceptional collections for the  equivariant derived category of $U_\Sigma$ with that of $\Aa^n.$
The set of objects for $U_\Sigma$ is a subset of that for $\Aa^n$ because sheaves $\FF_{I,\pp}$ restrict to zero unless
$I\in\Sigma.$ Moreover, some of the morphisms that were nonzero in $\Aa^n$ map to zero under the restriction functor. Specifically,
suppose we have nonzero $\Ext_{[U_\Sigma/\bT]}^*(\FF^{\Sigma}_{I,\pp},\FF^{\Sigma}_{J,\qq})$ and
$\Ext_{[U_\Sigma/\bT]}^*(\FF^{\Sigma}_{J,\qq},\FF^{\Sigma}_{K,\rr})$ which correspond to sets $S$ and $T.$
If $T$ is not contained in $I\backslash S$ then we definitely have
$\Ext_{[U_\Sigma/\bT]}^*(\FF^{\Sigma}_{I,\pp},\FF^{\Sigma}_{K,\rr})=0$ by the same argument as for $\Aa^n.$
However, the presence of $\Sigma$ gives an additional condition $I\cup K\in \Sigma.$
It can happen that $I\cup J\in \Sigma,\; J\cup K\in \Sigma,$ but $I\cup K\not \in \Sigma,$ in which case the composition
must be zero in $\bD^b(\Coh[U_\Sigma/\bT])$ whereas it was nonzero in $\bD^b(\Coh[\Aa^n/\bT]).$ We argue that every such zero
composition can be explained by the fact that it can be rerouted in $\bD^b(\Coh[\Aa^n/\bT])$ through an object that restricts to zero. Indeed, under the assumptions above, we have a composition of nonzero morphisms in $\bD^b(\Coh[\Aa^n/\bT])$
\[
\phi_{I,\pp; S\sqcup T,K}=\phi_{I\cup K,\pp; S\sqcup T, K}\circ\phi_{I, \pp; \emptyset, I\cup K},
\]
where the intermediate object ${\FF}_{I\cup K, \pp}[\bp]$ restricts to zero on $U_\Sigma.$
More generally, any morphism $\phi_{I,\pp; S,J}$ on $\Aa^n$ with $I\cup J\not\in\Sigma$ can be rerouted through
$\FF^{\Sigma}_{I\cup J,\pp}$ which restricts to zero on $U_\Sigma.$

We conclude this section by restating the result of Theorem \ref{main2} in the spirit of Theorem \ref{main1.5}.

\begin{theorem}\label{main2.5}
The derived category $\bD^b(\Coh[U_\Sigma/\bT])$ of $\bT$\!--equivariant coherent sheaves on $U_\Sigma$ possesses a strong full exceptional collection of objects $E_{a}$ indexed by $a=(a_1,\ldots,a_n)\in\Zz^n$ such that the set of $i$ for which $a_i\neq 0$ lies in $\Sigma.$
The objects $E_a$ are ordered lexicographically.
The morphism space $\Hom(E_a,E_b)$ is isomorphic to $\kk$ if and only if for all $i$ either $b_i-a_i\in \{0,1\}$ or $(b_i,a_i) = (1,-1)$ and, moreover, the set of $i$ for which $a_i$ or $b_i$ are nonzero lies in $\Sigma.$ All other morphism spaces are zero. 
The nontrivial spaces $\Hom(E_a,E_b)$ come equipped with a distinguished generator $\rho_{a,b}.$
The composition of morphisms $\Hom(E_a,E_b)\otimes \Hom(E_b,E_c)\to \Hom(E_a,E_c)$ is nontrivial if and only if all three spaces are nontrivial. In these cases $\rho_{b,c}\circ\rho_{a,b}=\rho_{a,c}.$
\end{theorem}

\begin{proof}
We reindex $\FF_{I,\pp}$ with $I\in\Sigma$ according to \eqref{rules} and \eqref{rulesback} as in the proof of Theorem \ref{main1.5}. The resulting indices $a$ are precisely the ones with $\{i,a_i\neq 0\}\in\Sigma.$ The rest follows from Theorem \ref{main2}.
\end{proof}

\section{Applications to toric varieties and stacks}\label{sec4}
We will now apply the results of the previous section to the equivariant 
categories of smooth toric Deligne-Mumford stacks.
For simplicity, we only consider the toric DM stacks that have trivial 
stack structure at the generic point.
The most convenient way to describe such toric stacks is via the 
\emph{extended} stacky fan of~\cite{Jiang}
(another possible choice is a stacky fan of~\cite{BCS}).

An extended stacky fan is a triple ${\bf\Sigma} = (N,\Sigma,\{v_i\})$
consisting of a finitely generated free abelian group $N,$ a rational 
simplicial fan $\Sigma$ in $N_\Rr,$
and a choice of lattice points $v_i,$ $i\in [n],$ such that each ray in 
$\Sigma$ contains one  of these points
(this identifies the fan $\Sigma$ with a simplicial complex on the set 
$[n],$ see \cite{Jiang} for details).
Furthermore, with this approach we can (and will) assume that the 
integer span of $v_i$ is all of $N.$

The toric stack $\Pp_{\Sigma}$ is defined as the stack quotient of the open subspace $U_\Sigma$ of $\Aa^n$ considered in the previous section by the algebraic subgroup $G_{\Sigma}$ of $\Gm^n$ which is the kernel of the surjective map
\[
\Gm^n \to \Hom(N^\vee,\Gm)
\quad
\text{given by}
\quad (\lambda_1,\ldots,\lambda_n) \mapsto \{ m\to \prod_{i=1}^n \lambda_i^{m\cdot v_i}\}.
\]
The DM stack $\Pp_{\Sigma}$ admits a natural action of the torus $\bT_{\Pp_\Sigma}=\Gm^n/G_{\Sigma}=\Hom(N^\vee,\Gm).$
This stack is a scheme (and is then a smooth toric variety) if and only if for every cone $\sigma\in \Sigma,$ the chosen points $v_i$ on rays of $\sigma$ span a saturated subgroup of $N.$

Our initial motivation behind this paper was to understand the abelian category of $\bT_{\Pp_\Sigma}$\!--equivariant coherent sheaves on $ \Pp_{\Sigma}$ and the corresponding bounded derived category. Our first observation is that the abelian category in question depends only on the simplicial complex. Indeed, it is equivalent to the category of $\Gm^n$\!--equivariant
sheaves on $U_\Sigma$ by construction, because
\[
[\Pp_{\Sigma}/\bT_{\Pp_\Sigma}] = [(U_\Sigma/G_\Sigma)/\bT_{\Pp_\Sigma}]
=[U_\Sigma/\Gm^n].
\]

As a consequence, Proposition \ref{exc_col} of the previous section gives full strong exceptional collections in the equivariant derived categories of toric varieties.
\begin{proposition}\label{exc_col_toric}
The equivariant toric derived category $\bD^b[\Pp_{\Sigma}/\bT_{\Pp_\Sigma}]$ has a full strong exceptional collection made of (pushforwards of) shifts of invertible sheaves on closures of toric strata. Specifically, the elements of the collection are the shifts of the equivariant sheaves $\FF^\Sigma_{I,\pp}[\bp],$ where $I$ runs over all subsets of $\Sigma$ and $\pp:[n]\to\Zz$ runs over all functions with the property $\pp(i)=0$
for all $i\not\in I.$
\end{proposition}

\begin{proof}
This is just a reformulation of Proposition \ref{exc_col}.
\end{proof}

\begin{remark}
{\rm
It is instructive to observe that the equivariant abelian category of a toric variety does not detect the dimension of the variety. For example,
when the toric variety is $\Gm^k,$ i.e. the simplicial complex consists of the empty set only, the resulting stack $[\Pp_{\Sigma}/\Gm^k]$ is a point. More generally, the derived category of $[\Pp_{\Sigma}/\bT_{\Pp_\Sigma}]$ is equal to that of $[U_\Sigma/\Gm^{\dim U_\Sigma}],$ but the spaces $\Pp_{\Sigma}$ and $U_\Sigma$ are usually of different dimension.
}
\end{remark}

We also note that the abelian category of equivariant coherent sheaves on a weighted projective space with weights $w_1,\ldots, w_k$ is equivalent to the category of $\Gm^k$\!--equivariant coherent sheaves on $\Aa^k\setminus\{0\}$ for any weights with $\gcd(w_1,\ldots,w_k)=1.$ In fact, we will calculate the $\Ext$ spaces between the equivariant line bundles on the weighted projective space in the following lemma which will be used later.
\begin{lemma}\label{trivial}
For the functions $\pp,\qq:\{1,\ldots,k\}\to \Zz$ and the corresponding equivariant line bundles on the weighted projective space of dimension $(k-1)$ there holds
\[
\Ext^{j}_{[\Pp^{w_1,\ldots,w_k}/\Gm^{k-1}]}(\FF(\pp),\FF(\qq)) =
\left\{
\begin{array}{ll}
\kk,&{\rm if~} j=0,~{\rm and~} \qq(i)\geq \pp(i)~{\rm for~all~}i\\
\kk,&{\rm if~} j=k-1,~{\rm and~} \qq(i)< \pp(i)~{\rm for~all~}i\\
0,&{\rm otherwise.}
\end{array}
\right.
\]
\end{lemma}

\begin{proof}
The $\Ext$ spaces in question are the cohomology spaces of $\FF(\qq-\pp).$ As usual, we calculate these spaces using \v{C}ech cover by open affine subsets $U_i$ characterized by $z_i\neq 0.$ The degree zero part of the term of the \v{C}ech complex that corresponds to $I\subseteq \{1,\ldots,k\}$ is $\kk$ if and only if $(\qq-\pp)\vert_{\bar I}\geq 0.$ 
If we augment the \v{C}ech complex to include the term with $I=\emptyset,$ we  observe that its
degree zero part is acyclic as long as  there exists $i$ with $\qq(i)\geq \pp(i).$ So in this case  the degree zero part of the cohomology of the \v{C}ech complex is equal to the (missing) $I=\emptyset$ term and contributes to  $\Ext^0.$  If $\qq(i)<\pp(i)$ for all $i,$ then the only nontrivial degree zero term of the \v{C}ech complex occurs at $I=[k]$ and contributes $\kk$ to $\Ext^{k-1}.$ 
\end{proof}

The equivariant derived categories of toric varieties depend even less on the varieties themselves. To illustrate this phenomenon, let us start with the following example.

\begin{proposition}\label{bla2}
The $\Gm^2$\!--equivariant derived category of the affine plane $\Aa^2$ is equivalent to the  $\Gm^2$\!--equivariant derived category of the blowup $\widetilde{\Aa^2}$ of $\Aa^2$ at the origin.
\end{proposition}

\begin{proof}
The derived category of $[\Aa^2/\Gm^2]$ is generated by the line bundles $\FF(a_1,a_2)$ for $a_1,a_2\in\Zz$ where we identify $(a_1,a_2)$ and the corresponding function $\aa:\{1,2\}\to \Zz.$

The derived category of $[\widetilde{\Aa^2}/\Gm^2]$ is equivalent to the derived category of
$[U_\Sigma/\Gm^3]$ where $\Sigma$ is the simplicial complex in $\{1,2,3\}$ which consists of
\[
\{1,3\},\{2,3\}, \{1\},\{2\},\{3\},\emptyset.
\]
In other words we have $U_\Sigma=\Aa^3\setminus\{z_1=z_2=0\}.$

The derived category of $[U_\Sigma/\Gm^3]$  is generated by the line bundles $\FF^\Sigma(a_1,a_2,a_3).$
We can calculate the $\Ext$ spaces between these line bundles by looking at the \v{C}ech complex with two open affine subsets
$U_1$ and $U_2$ characterized by $z_1\neq 0$ and $z_2\neq 0$ respectively. As in Lemma \ref{trivial}, we see that
\[
\Hom(\FF^\Sigma(a_1,a_2,a_3),\FF^\Sigma(b_1,b_2,b_3))=\kk
\]
if $b_i\geq a_i$ for all $i,$ and
\[
\Ext^1(\FF^\Sigma(a_1,a_2,a_3),\FF^\Sigma(b_1,b_2,b_3))=\kk
\]
if $b_3\geq a_3,$ $b_1<a_1,$ $b_2<a_2.$ All the other $\Ext$ spaces are zero.

We observe that the line bundles
\[\FF^\Sigma(a_1,a_2, \left\lfloor \frac {a_1+a_2}2 \right\rfloor)
\]
form a full strong exceptional collection. Specifically, we have
\[
\Hom(\FF^\Sigma(a_1,a_2, \left\lfloor \frac {a_1+a_2}2 \right\rfloor),\FF^\Sigma(b_1,b_2, \left\lfloor \frac {b_1+b_2}2 \right\rfloor))=\kk
\]
if and only if $b_1\geq a_1$ and $b_2\geq a_2.$ All the $\Ext$ groups between these line bundles are zero, because
$ \left\lfloor \frac {a_1+a_2}2 \right\rfloor\geq  \left\lfloor \frac {b_1+b_2}2 \right\rfloor$ implies that $a_1\geq b_1$ or $a_2\geq b_2.$
The collection is full exceptional, because the exact sequences on $U_\Sigma$
\[
0\to \FF^\Sigma(a_1-1,a_2-1,a_3)\to  \FF^\Sigma(a_1,a_2-1,a_3) \oplus  \FF^\Sigma(a_1-1,a_2,a_3) \to  \FF^\Sigma(a_1,a_2,a_3) \to 0
\]
can be used to generate the line bundles $\FF^\Sigma(a_1,a_2,c)$ for a fixed value of $c$ starting with two diagonal lines with
$a_1+a_2=2c$ and $a_1+a_2=2c+1.$

We note that the $\Hom$ spaces between the above bundles are canonically isomorphic to the $\Hom$ spaces between the corresponding
line bundles $\FF(a_1,a_2)$ on $[\Aa^2/\Gm^2],$ which provides the desired equivalence.
\end{proof}

\begin{remark}
{\rm
The conceptual meaning of the above equivalence is the following. Consider the following two stacky fans in $\Zz^2.$
The fan $\Sigma_1$ has vertices $v_1=(1,1)$ and $v_2=(-1,1),$ and the simplicial complex with maximum set $\{1,2\}.$ The corresponding toric DM stack $\Pp_{\Sigma_1}$  is $[\Aa^2/\Zz_2].$ The second fan $\Sigma_2$ has vertices $v_1=(1,1),$ $v_2=(-1,1),$ and $v_3=(0,1)$ with the two maximum sets $\{1,3\}$ and $\{2,3\}.$ The corresponding toric variety $\Pp_{\Sigma_2}$ is the resolution of $A_1$ singularity.
It is well known, see for example the work of Kawamata \cite{Kawamata}, that $\Pp_{\Sigma_1}$  and $\Pp_{\Sigma_2}$ are derived equivalent. It turns out that the functor naturally extends to the equivariant setting. Then it remains to notice that the (abelian) equivariant categories of coherent sheaves on  $[\Aa^2/\Zz_2]$ and the resolution of $A_1$ singularity are equivalent to those of $\Aa^2$ and $\widetilde{\Aa^2}$ respectively.
}
\end{remark}

We can generalize the above construction to a more general blowup (see \cite{BKF}). We will state the result in terms of simplicial complexes, where it is usually called a stellar subdivision, see \cite{Lickorish}.
Let $\Sigma$ be a simplicial complex on the set $\{1,\ldots, n\}$ and fix $\sigma \in \Sigma$  of size at least two. We introduce the blowup of $\Sigma$ at $\sigma$
as a simplicial complex  on $\{1,\ldots, n,n+1\}$ by replacing every maximum set $\tau\in\Sigma$ that contains $\sigma$ by $|\sigma|$ sets
\[
\Big\{\tau\cup\{n+1\}\setminus\{i\}\Big\}_{i\in \sigma}
\]
and taking all of their subsets. Another way to define the complex $\widetilde\Sigma$ is by saying that $I\in \widetilde\Sigma$ if and only if
\begin{equation}\label{wts}
(n+1)\not\in I, ~I\in\Sigma,~\sigma\not\subseteq I,~{\rm ~or~}~(n+1)\in I,~I\setminus\{n+1\}\in\Sigma, (I\setminus\{n+1\})\cup\sigma\in \Sigma,
~\sigma\not\subseteq I.
\end{equation}
%We will also consider the map $r:\tilde\Sigma\to \Sigma$
%\begin{equation}\label{mapr}
%r(I)=\left\{
%\begin{array}{ll}
%I ,&{\rm if~}(n+1)\not\in I\\
%\sigma \cup I\setminus \{n+1\} ,&{\rm if~}(n+1)\in I
%\end{array}
%\right.
%\end{equation}
%which corresponds to looking at the minimum cone of the geometric realization of $\Sigma$ that contains a cone of
%$\widetilde\Sigma.$

We define a functor from $\bD^b(\Coh[U_\Sigma/\Gm^n])$ to $\bD^b(\Coh[U_{\widetilde\Sigma}/\Gm^{n+1}])$
as follows. There is a map $\pi:U_{\widetilde\Sigma}\to U_\Sigma$ given by
\begin{equation}\label{blp}
\pi^* z_i = \left\{\begin{array}{ll}
\tilde z_i,& i\not\in\sigma\\
\tilde z_i\tilde z_{n+1},& i \in\sigma
\end{array}
\right.
\end{equation}
with the induced map on the open subsets $\Gm^{n+1}\to\Gm^n$ which is a group homomorphism. This gives a map of stacks
and the corresponding pullback in derived categories
 $\pi^*:\bD^b(\Coh[U_\Sigma/\Gm^n])\to \bD^b(\Coh[U_{\widetilde\Sigma}/\Gm^{n+1}]).$
Note that it sends $\FF(\pp)$ to $\FF(t(\pp))$ with $t(\pp)(n+1)=\sum_{i\in\sigma} \pp(i)$ and $t(\pp)(i)=\pp(i)$ for $i<n+1.$

For each integer $k>0$ there is a map $\zeta_k: U_{\widetilde\Sigma}\to U_{\widetilde\Sigma}$ given by
\[
\zeta_k^* \tilde z_i = \left\{\begin{array}{ll}
 \tilde z_i,& i\neq n+1\\
\tilde z_{n+1}^{k},& i=n+1.
\end{array}
\right.
\]
This map gives a  group homomorphism $\Gm^{n+1}\to \Gm^{n+1}$ which 
fits into a short exact sequence
\begin{align}\label{muk}
1\to {\mathbb \mu}_k\to \Gm^{n+1}\to \Gm^{n+1}\to 1.
\end{align}
For every equivariant sheaf $F$ its pushforward $\zeta_{k*}F$ is 
equivariant with respect to the induced action of $\Gm^{n+1}.$ After taking ${\mathbb \mu}_k$ invariants (at every invariant open set), it becomes equivariant with respect to the action of
the target (second copy in \eqref{muk}) $\Gm^{n+1}.$ We abuse the notation to denote this sheaf of invariants by
$\zeta_{k*}F(U).$ In other words, sections of $\zeta_{k^*}F(U)$ on a $\Gm^{n+1}$\!--invariant open set $U$ are the sections of $F(U)$ whose characters have last coordinate divisible by $k.$
This functor is exact and induces a functor between the derived categories that we also denote by $\zeta_{k*}.$
Note that $\zeta_{k*}\FF(\pp)$ is naturally isomorphic to  $\FF(\pp')$ where $\pp'(n+1) = \lfloor \frac {\pp(n+1)} k \rfloor$ and
$\pp'(i)=\pp(i)$ for $1\leq i\leq n.$

We set $k=|\sigma|$ and consider the composition functor 
\[
\zeta_{|\sigma|*} \pi^* : \bD^b(\Coh[U_\Sigma/\Gm^n])\lto \bD^b(\Coh[U_{\widetilde\Sigma}/\Gm^{n+1}]).
\]

\begin{theorem}\label{main3}
Let simplicial complex $\widetilde\Sigma$ be obtained from the simplicial complex $\Sigma$ by the stellar subdivision at $\sigma\in\Sigma.$
Then the composition functor $\zeta_{|\sigma|*} \pi^*: \bD^b(\Coh[U_\Sigma/\Gm^n])\to \bD^b(\Coh[U_{\widetilde\Sigma}/\Gm^{n+1}])$ is an equivalence.
\end{theorem}

\begin{proof}
The derived category of $[U_\Sigma/\Gm^n]$ is generated by the line bundles $\FF^\Sigma(\pp)$ for $\pp:\{1,\ldots,n\}\to \Zz.$
We will prove that the maps
\begin{equation}\label{bigclaim}
\Ext^*(\FF^\Sigma(\pp), \FF^\Sigma(\qq))
\lto
\Ext^*(\zeta_{|\sigma|*}\pi^* \FF^\Sigma(\pp), \zeta_{|\sigma|*}\pi^* \FF^\Sigma(\qq))
\end{equation}
are isomorphisms for all $\pp$ and $\qq,$ and that the objects $\zeta_{|\sigma|*} \pi^* \FF^\Sigma(\pp)$ generate the whole category
$\bD^b(\Coh[U_{\widetilde\Sigma}/\Gm^{n+1}]).$

We will first prove that the set of objects of the form  $\zeta_{|\sigma|*} \pi^* \FF^\Sigma(\pp)$  generate the whole category.
As was observed earlier,  $\zeta_{|\sigma|*} \pi^* \FF^\Sigma(\pp) \cong \FF^{\widetilde\Sigma}(s(\pp)),$ where
\[
s(\pp)(n+1) = \left\lfloor \frac {\sum_{i\in\sigma} \pp(i)}{|\sigma|} \right\rfloor
\]
and $s(\pp)(i)=\pp(i)$ for all $1\leq i\leq n.$  Therefore, the subcategory $\DD\subseteq\bD^b(\Coh[U_{\widetilde\Sigma}/\Gm^{n+1}])$ generated by these line bundles contains
$\FF(\qq)$ for all $\qq:\{1,\ldots,n+1\}\to \Zz$ with
\[
\sum_{i\in\sigma}\qq(i) - |\sigma| \qq(n+1) \in \{0,\ldots, |\sigma|-1\}.
\]
Observe that since $\sigma\not\in\widetilde\Sigma,$ there is a long exact sequence  of length $|\sigma|+1$ of sheaves on $U_{\widetilde\Sigma}$
\[
0\lto \FF^\Sigma(- \cchi_\sigma)\lto \cdots \lto \bigoplus_{J\subseteq \sigma, |J|=l} \FF^\Sigma(-\cchi_J)\lto \cdots
\lto \bigoplus_{i\in \sigma} \FF^\Sigma(-\ee_i)\lto \FF^\Sigma\lto 0
\]
coming from the Koszul complex for the variables $\tilde z_i,~i\in\sigma.$ If we twist this sequence by $\FF(\qq)$ with
$\sum_{i\in\sigma}\qq(i) - |\sigma| \qq(n+1) = |\sigma|$ then all but the last term are in the subcategory $\DD.$ Therefore, all
such $\FF(\qq)$ lie in the subcategory $\DD$ as well. We can similarly extend the range of $\sum_{i\in\sigma}\qq(i) - |\sigma| \qq(n+1) $
down by one and repeat the process to get all $\FF(\qq).$

So we only need to verify that the maps in \eqref{bigclaim} are indeed isomorphisms.
These maps are the compositions of
\begin{equation}\label{mappi}
\Ext^*(\FF^\Sigma(\pp), \FF^\Sigma(\qq))
\lto
\Ext^*(\pi^* \FF^\Sigma(\pp), \pi^* \FF^\Sigma(\qq))
\end{equation}
and
\begin{equation}\label{maprho}
\Ext^*(\pi^* \FF^\Sigma(\pp), \pi^* \FF^\Sigma(\qq))
\lto
\Ext^*(\zeta_{|\sigma|*}\pi^* \FF^\Sigma(\pp), \zeta_{|\sigma|*} \pi^* \FF^\Sigma(\qq)).
\end{equation}
We observe that the maps of \eqref{mappi} are isomorphisms because they are simply the $\Gm^{n}$\!--invariant parts of the maps induced by the derived pullback  for the map $U_{\widetilde\Sigma}/\Gm\to U_\Sigma,$ where we take the quotient by the free action of $\Gm,$ which is the kernel of the map of \eqref{blp}. This is precisely the blowup of $U_\Sigma$ at the smooth closed subvariety $Z_\sigma$ given by
\[
Z_\sigma = U_\Sigma \cap \Big(\bigcap_{i\in\sigma}\{ z_i = 0\}\Big)=U_\Sigma\setminus \Big(\bigcup_{\tau\nsupseteq\sigma}U_\tau\Big).
\]
It is well known that the derived pullback of a blowup is fully faithful in the derived category of non-equivariant coherent sheaves (see, e.g., \cite{Or}).
When we apply this statement to $\Ext^*(\FF^\Sigma(\pp),\FF^\Sigma(\qq))$ and take $\Gm^{n}$ invariants  we see that
the maps of \eqref{mappi} are isomorphisms.

Let us now investigate the maps of \eqref{maprho}. The $\Ext^*$ spaces on the left can be computed by the \v{C}ech complex for the
open affine subsets $U_\tau$ of $U_{\widetilde\Sigma}$ for all maximum $\tau\in\widetilde\Sigma.$ The component for a nonempty
set $T$ of these maximum subsets is given by
the $\Gm^{n+1}$ invariant parts of
\[
\Hom(\FF(t(\pp))\Big\vert_{\!\!\mathop{\bigcap}\limits_{\tau\in T}\!U_\tau},\;
\FF(t(\qq))\Big\vert_{\!\!\mathop{\bigcap}\limits_{\tau\in T}\!U_\tau}),
\quad\text{where}\quad
t(\pp)(n+1)=\sum_{i\in\sigma} \pp(i),\; t(\pp)(i)=\pp(i),\; i\le n
\]
 and similarly for $t(\qq).$ The $\Hom$ space in question is isomorphic to the product of polynomial rings in $\tilde z_i$ for $i\in \mathop{\bigcap}\limits_{\tau\in T} \tau$ and Laurent polynomial rings in $\tilde z_i$ for $i\not\in \mathop{\bigcap}\limits_{\tau\in T} \tau$ with the grading of the generator shifted by $t(\pp)-t(\qq).$ As a consequence, the $T$ component
is isomorphic to $\kk$ if
$t(\qq)\geq t(\pp)$ on $\mathop{\bigcap}\limits_{\tau\in T} \tau$ and is zero otherwise.

The spaces on the right are obtained by looking at the graded components of the sheaves $\FF(t(\pp))\Big\vert_{\!\!\mathop{\bigcap}\limits_{\tau\in T}\!U_\tau}$ and $\FF(t(\qq))\Big\vert_{\!\!\mathop{\bigcap}\limits_{\tau\in T}\!U_\tau},$ where the last coordinate is divisible by $|\sigma|.$ They are equal to
$\kk$ if $s(\qq)\geq s(\pp)$ on $\mathop{\bigcap}\limits_{\tau\in T} \tau$ and are zero otherwise. The chain map between the two \v{C}ech complexes preserves the components. If both spaces are $\kk$ then it is nonzero, otherwise it is zero.

We will show that this chain map of complexes is a quasi-isomorphism. Note that if

\[
t(\qq)(n+1)-t(\pp)(n+1)=\sum_{i\in\sigma}(\qq(i)-\pp(i))\quad \text{and}\quad
s(\qq)(n+1)-s(\pp)(n+1)=
\left\lfloor \frac{\sum_{i\in\sigma}\qq(i)}{|\sigma|} \right\rfloor - \left\lfloor \frac{\sum_{i\in\sigma}\pp(i)}{|\sigma|} \right\rfloor
\]
are both negative or are both nonnegative, then the chain map is an isomorphism at all $T$ components, and the statement is clear.
We also observe that $t(\qq)(n+1)\geq t(\pp)(n+1)$ implies  $s(\qq)(n+1)\geq s(\pp)(n+1),$ so we only need to consider the
case when $t(\qq)(n+1)< t(\pp)(n+1)$ and $s(\qq)(n+1)-s(\pp)(n+1)\geq 0.$
The first inequality implies that $\qq(i)<\pp(i)$ for some $i\in\sigma.$ However, if we had  $\qq(i)<\pp(i)$ for all $i\in\sigma,$ then we would have $t(\qq)\leq t(\pp)-|\sigma|$ which would lead to $s(\qq)(n+1)< s(\pp)(n+1).$
This means, in particular, that we only need to consider the case
\begin{equation}\label{case}
\exists\; i, j\in\sigma\quad {~\rm such~that~}\quad \qq(i)-\pp(i)\geq 0,\quad\text{and}\quad \qq(j)-\pp(j)<0.
\end{equation}

Since $t(\qq)(n+1)< t(\pp)(n+1)$ and $s(\qq)(n+1)-s(\pp)(n+1)\geq 0,$ the map of complexes is injective. The cokernel  complex $C$ has $T$\!--components equal to $\kk$ if and only if all $\tau\in T$ contain $n+1$ and
$s(\qq)\geq s(\pp)$ on $\mathop{\bigcap}\limits_{\tau\in T}\tau.$  From now on we can ignore the elements of $\widetilde\Sigma$ that do not contain $n+1.$  This allows us to observe that $C$ is the $\Gm^{n+1}$\!--invariant part of the \v{C}ech complex that computes the cohomology of the restriction of the line bundle $\FF(s(\qq)-s(\pp))\Big|_{D}$ to the divisor $D=\{\tilde z_{n+1}=0\}$ in $U_{\widetilde\Sigma}.$

We observe that $D$ is the exceptional divisor and has a $\Gm$\!--invariant morphism to the center of the blowup $Z_\sigma\subset U_\Sigma.$ Moreover, the morphism $p: D\to Z_\sigma$ is $\Pp^{|\sigma|-1}$\!--bundle
over $Z_{\sigma}.$ Now we can see that the derived direct image ${\mathbf R} p_{*} (\FF(s(\qq)-s(\pp))|_{D})$ is zero. Indeed, since the property \eqref{case} holds for the line bundle $\FF(s(\qq)-s(\pp)),$ its restriction
on the fibers of $p$ have trivial cohomology by Lemma \ref{trivial}. Thus the cohomology of $\FF(s(\qq)-s(\pp))\Big|_{D}$ are also trivial.
\end{proof}

Theorem \ref{main3} has interesting consequences. 
As we discussed in the introduction, to any simplicial complex $\Sigma$ one can associate a topological space $|\Sigma|$ and then study PL-homeomorphism classes of $|\Sigma|.$
The following result indicates that $\bD^b(\Coh[U_{\Sigma}/\bT])$ is an invariant of PL-homeomorphism classes.
\begin{theorem}\label{main4}
Let $\Sigma_1$ and $\Sigma_2$ be two simplicial complexes.
If polyhedra $|\Sigma_1|$ and $|\Sigma_2|$ are PL-homeomorphic, then the  derived categories
$\bD^b(\Coh[U_{\Sigma_1}/\bT_1])$ and $\bD^b(\Coh[U_{\Sigma_2}/\bT_2])$
 of torus-equivariant coherent sheaves on $U_{\Sigma_1}$ and $U_{\Sigma_2}$ respectively are equivalent.
\end{theorem}

\begin{proof}
It is a (non-trivial) known result that PL-homeomorphic simplicial complexes can be connected by stellar subdivisions and their inverses which are called stellar welds  (see, e.g., \cite{Lickorish}). This means that there 
exist simplicial complexes
\[
\Sigma_1 = \Sigma_1', \Sigma_2',\ldots,\Sigma_k'=\Sigma_2
\]
such that for every $ 1\leq i<k$ the complex $\Sigma_{i+1}'$ is a stellar subdivision of $\Sigma_i'$ or vice versa. By Theorem \ref{main3}, the categories 
$\bD^b(\Coh[U_{\Sigma_{i}'}/\bT_i'])$ and $\bD^b(\Coh[U_{\Sigma_{i+1}'}/\bT_{i+1}'])$ are equivalent
which leads to the desired equivalence of $\bD^b(\Coh[U_{\Sigma_1}/\bT_1])$ and $\bD^b(\Coh[U_{\Sigma_2}/\bT_2]).$
\end{proof}

\begin{corollary}
All proper toric Deligne-Mumford stacks of a given dimension $k$ have equivalent bounded derived categories of $\Gm^k$\!--equivariant coherent sheaves.
\end{corollary}

\begin{proof}
The simplicial complexes in question are clearly PL-homeomorphic to each other and to the $(k-1)$\!--dimensional sphere with the natural PL-structure on it.
\end{proof}

We would like to point out some interesting questions.
\begin{question}
{\rm
Are derived categories  $\bD^b(\Coh[U_\Sigma/\Gm^n])$ invariants of homeomorphisms or only PL-homeomorphisms? In other words, if two simplicial spaces are homeomorphic but not  PL-homeomorphic, what can be said about these categories? One can find examples of such simplicial spaces in \cite{Edwards}.
}
\end{question}

\begin{question}
{\rm
What can be said about the group of auto-equivalences of  $\bD^b(\Coh[U_\Sigma/\Gm^n])$ or its image in the group of automorphisms of the Grothendieck group?}
\end{question}

\section{Additional remarks and discussion}\label{sec.5}
\subsection{Full exceptional collections of line bundles.}
It is interesting to ask whether every equivariant derived category of a toric DM stack admits a full exceptional collection made from (shifts of) equivariant line bundles. This can be viewed as an analog of King's conjecture, which has been studied in great detail, see for example \cite{Efimov,HillePerling,BorisovHua}.

If such collections exist in the non-equivariant setting, they can be lifted to the equivariant one. For example, if we  take the usual collection $\langle\FF,\FF(1),\FF(2)\rangle$ on $\Pp^2$ we obtain the collection
$$
\FF(a_1,a_2,a_3)  {\rm ~with~} 0\leq a_1+a_2+a_3\leq 2
$$
on $\Aa^3\setminus\{0\}.$ Similarly, the usual full exceptional collection of line bundles on the weighted projective space with weights $w_1,\ldots,w_n$  lifts
to the collection
$$
\{\FF(a_1,\ldots,a_n),~0\leq \sum_{i=1}^n w_i a_i < \sum_{i=1}^n w_i\}
$$
of equivariant line bundles on $\Aa^n\setminus\{0\}.$ We are assuming that the weights have no common divisor, so that the corresponding toric stack has trivial stack structure at the generic point.

While we suspect that most toric varieties and stacks do not possess full exceptional collections of line bundles, we were able to construct some rather unexpected examples. Specifically, we considered $\Pp^2$ without three torus fixed points. This is a toric variety given in the lattice $\Zz^2$ by the fan $\Sigma$ that consists of three one-dimensional cones
generated by $v_1=(1,0),$ $v_2=(0,1)$ and $v_3=(-1,-1)$ and the zero-dimensional cone.
The corresponding simplicial complex is $\Sigma=\{\{1\},\{2\},\{3\},\emptyset\}.$

This fan gives rise to the subset $U_\Sigma\subset \Aa^3$ with the property that at most one coordinate of $(z_1,z_2,z_3)$ can be zero. The category of equivariant sheaves on $\Pp^2$ without the zero-dimensional orbits is equivalent to the category of
$\Gm^3$\!--equivariant sheaves on $U_\Sigma.$ The equivariant line bundles are indexed by $\pp:\{1,2,3\}\to\Zz.$ The following proposition describes the $\Ext$ spaces between the sheaves $\FF^\Sigma(\pp_i).$

\begin{proposition}\label{weirdExt}
Let $0\leq m\leq 3$ be the number of nonnegative coordinates of $(\pp_2-\pp_1).$ Then $\Ext^s(\FF^\Sigma(\pp_1),\FF^\Sigma(\pp_2))=0$ unless $(m,s)=(3,0),$ $(m,s)=(1,1),$ or $(m,s)=(0,1).$ In the last case the dimension of the $\Ext^1$ is two.
\end{proposition}

\begin{proof}
It suffices to calculate the degree zero part of the cohomology of $\FF^\Sigma(\pp)$ with the number of
nonnegative coordinates of $\pp$ equal to $m.$
Consider the  ${\rm \check Cech}$ complex for the open cover with the open affine subsets
\[
U_{1,2}=\{z_1\neq 0,z_2\neq 0\},~U_{1,3}=\{z_1\neq 0,z_3\neq 0\},~U_{2,3}=\{z_2\neq 0,z_3\neq 0\},~
\]
whose pairwise intersections are $\Gm^3.$ The sections of $\FF^\Sigma_{\pp}$ on $U_{1,2}$ are the space
\[
\kk[z_1^{\pm 1},z_2^{\pm 1},z_3](\pp).
\]
Therefore, the degree zero part is $\kk$ if $\pp(3)\geq 0$ and is $0$ if $\pp(3)<0.$ Similar statements hold for the other $U_{i,j},$ and their intersections. So we get
the degree zero part of the ${\rm \check Cech}$ complex to be
\[
0\to \kk^m \to \kk^3\to \kk\to 0.
\]
The second map is always surjective, and the first map is easily seen to be injective for $m\leq 2,$ which implies the claim of the proposition.
\end{proof}

\begin{proposition}
Consider $U_\Sigma=\Aa^3\setminus\mathop{\bigcup}\limits_{i\neq j}\{z_i=z_j=0\}.$
The following ordered set $S$ of points $\pp\in\Zz^3$
\[
\{(0,0,k),k\ge 0\} \prec \{(a,-a,0), a\neq 0\} \prec \{(a,1-a,0), a\neq 1\} \prec \{(1,0,-k), k\ge 0\}
\]
gives a full exceptional collection of line bundles $\FF^\Sigma(\pp)$ in the derived category $\bD^b(\Coh[U_\Sigma/\Gm^3]).$

The ordering among $(0,0,k)$ is by increasing $k,$ and  the ordering among $(1,0,-k)$ is by decreasing $k.$ The ordering among $(a,-a,0)$ is arbitrary, as is the ordering among $(a,1-a,0).$
\end{proposition}

\begin{proof}
As a consequence of Proposition \ref{weirdExt}, a line bundle $\FF^\Sigma(\pp)$ on the stack $[U_\Sigma/\Gm^3]$ is acyclic
if and only if $\pp$ has
two nonnegative coordinates. Therefore, an ordered collection of line bundles is exceptional if and only if the corresponding subset of $\Zz^3$ satisfies the property that  $\pp_1-\pp_2$ has exactly two nonnegative entries for all $\pp_1\prec \pp_2.$

Let us check that the above collection is exceptional. We observe that every element $\pp$ other than $(0,0,0)$ has exactly two nonpositive coordinates. Therefore, all $(-\pp)$ have two nonnegative coordinates, and we can claim the property for $\pp_1=(0,0,0).$
 The collection is symmetric with respect to $(\pp\mapsto (1,0,0)-\pp)$ so the same can be said for $\pp_2=(1,0,0).$ For $(0,0,k_1)=\pp_1\prec \pp_2 = (0,0,k_2)$ we see that $\pp_1-\pp_2=(0,0,k_1-k_2)$ has two nonnegative coordinates. For $(0,0,k)=\pp_1\prec \pp_2=(a,-a,0),a\neq 0$ we have $\pp_1-\pp_2=(-a,a,k)$ which has two nonnegative coordinates because $a\neq 0.$
For $(0,0,k)=\pp_1\prec \pp_2=(a,1-a,0)$ we have  $\pp_1-\pp_2=(-a,a-1,k)$ which has two nonnegative coordinates for any $a.$ For
$(0,0,k_1)=\pp_1\prec \pp_2 = (1,0,-k_2)$ we have $\pp_1-\pp_2=(-1,0,k_1+k_2)$ with two nonnegative coordinates.
For $\pp_1$ and $\pp_2$ both in the set $(a,-a,0),$ the difference is of the same form, with two nonnegative coordinates. For
$(a_1,-a_1,0)=\pp_1\prec \pp_2 = (a_2,1-a_2,0),$  the difference is $\pp_1-\pp_2=(a_3,-1-a_3,0)$ which has two nonnegative coordinates.
All the other cases are handled by the symmetry above.

To show that the above collection generates the derived category, observe that we have the following short exact sequences on $U_\Sigma$
\[
0 \lto \FF^\Sigma(\pp-(1,1,0)) \lto \FF^\Sigma(\pp-(1,0,0))\oplus \FF^\Sigma(\pp-(0,1,0))\lto \FF^\Sigma(\pp)\lto 0
\]
and similarly for the other pairs of coordinates. It means that if we have three points in a coordinate square, then the category would contain the line bundle that corresponds to the fourth point. We now see that using these coordinate squares we can get all of $\Zz^3$ from the set $S.$
Since $S$ contains all $(a,b,0)$ with $a+b\in\{0,1\},$ we can get  $(a,b,0)$ for all $a$ and $b.$ Then we use $\pp=(0,0,1)$ to get all $(a,b,1),$ then use $(0,0, 2)$ to get all $(a,b,2),$ and so on. We use $(1,0,-1)$ to get $(a,b,-1),$ then use $(1,0,-2)$ to get $(a,b,-2),$ and so on.
\end{proof}

\begin{remark}
{\rm
The above full exceptional collection can not be made into a strong exceptional one by cohomological shifts. Indeed, there are nonzero spaces
\[
\Hom (\FF^\Sigma(0,0,0), \FF^\Sigma(0,0,1) ), ~\Hom (\FF^\Sigma(0,0,0), \FF^\Sigma(0,1,0) ), ~\Hom (\FF^\Sigma(-1,1,0), \FF^\Sigma(0,1,0)),
\]
which means that these line bundles would need to come with the same shift. However, we have
\[
\Ext^1(\FF^\Sigma(0,0,1),\FF^\Sigma(-1,1,0)) \neq 0.
\]
}
\end{remark}

\begin{remark}
{\rm
If we similarly remove all codimension two coordinate subspaces from $\Pp^{n}$ for $n>2,$ then exceptional  collections of line bundles correspond to ordered subsets of $\Zz^{n+1}$ with the property that $\pp_1\prec \pp_2$ implies that $\pp_1-\pp_2$ has exactly $n$ nonnegative coordinates.
We do not currently know if one can find full exceptional collections of line bundles in this case.
}
\end{remark}

\subsection{From commutative to noncommutative toric stacks and back again}
It is also interesting to consider noncommutative toric stacks and to understand derived category of equivariant coherent sheaves with respect to the maximal torus on them.
We can show that these categories are actually equivalent to the corresponding categories on commutative stacks.

Let $L$ be a finitely generated abelian group and  $A=\bigoplus_{\pp\in L}A_{\pp}$ be an $L$\!--graded $\kk$\!--algebra.
Denote by $\GrL A$ the abelian category of all right $L$\!--graded modules over $A.$

With any such algebra $A$ we can associate a $\kk$\!--linear category $\tA$ whose set of objects is exactly $L$ and
whose spaces of morphisms $\tA_{\pp \qq}:=\Hom_{\tA}(\pp, \qq)$ are isomorphic to $A_{\qq-\pp}.$
The composition of morphisms in the category $\tA$ is induced by multiplication in the algebra $A.$
Note that the category $\tA$ is a full subcategory of the category $\GrL  A$ formed by free modules
$A(\pp)$ for all $\pp\in L.$ A category of the form $\tA$ is often called $L$\!--algebra (see, e.g., \cite{BGS}).

A right module $\tM$ over the category $\tA$ is a functor from the opposite category $\tA^{\op}$
to the category of $\kk$\!--vector spaces $\Mod k.$ Equivalently, a right $\tA$\!--module is given by a set of  $\kk$\!--vector spaces
$\{\tM_{\pp}\},$ for each $\pp\in L,$ and by morphisms of vector spaces
$
\tM_{\qq}\otimes \tA_{\pp\qq}\to \tM_{\pp}
$
compatible with compositions and units.
Right $\tA$\!--modules forms an abelian category which we denote by $\Mod \tA.$
The abelian category $\Mod \tA$ is naturally equivalent to the abelian category $\GrL  A$ of $L$\!--graded
right $A$\!--modules. A natural equivalence sends a right $L$\!--graded module
$M=\bigoplus_{\pp\in L} M_{\pp}$ to the right $\tA$\!--module $\tM,$ for which
$\tM_{\pp}=M_{-\pp}.$ In fact, this equivalence  is given by the functor
\[
\Hom_{\GrL  A}(\bigoplus_{\pp\in L} A(\pp), -): \GrL  A\stackrel{\sim}{\lto} \Mod \tA.
\]
In some sense this equivalence gives us a generalization of Proposition \ref{linebundles}.

Let $\bD(\GrL  A)$ be the unbounded derived category of all $L$\!--graded modules.
It is a triangulated category which admits arbitrary small coproducts (direct sums).
An object $C^{\cdot}\in \bD(\GrL  A)$ is called {\em compact} if $\Hom(C^{\cdot}, -)$ commutes with arbitrary small coproducts.
Compact objects in $\bD(\GrL  A)$ form a full triangulated subcategory
that is called the subcategory of perfect complexes (objects) and is denoted by $\perf A.$

\begin{remark}
{\rm
In the case when the algebra $A$ is noetherian and  has finite global dimension the category of perfect complexes $\perf A$ is equivalent
to
the bounded derived category of finitely generated modules $\bD^b(\grL A).$ 
}
\end{remark}

The set of free modules $\big\{A({\pp}),\;{\pp}\in L\big\}$ forms a set of compact generators for
the derived category $\bD(\GrL  A).$ This means that they are compact, i.e. $A(\pp)\in \perf A,$
and for any object $Z^{\cdot}\in \bD(\GrL  A)$ there are $A(\pp)$ and $n\in\Z$ such that $\Hom(A(\pp), Z^{\cdot}[n])\ne 0.$
Definition of compact generators is  closely related to the definition of classical generators.
Recall that a set of objects $S\subset \perf A$
forms a {\em set of classical generators} if the category $\perf A$ coincides with the smallest triangulated subcategory of
$\perf A$ which contains $S$ and is closed under taking direct summands.
It is proved in \cite{Ne1} that for any compactly generated triangulated category $\TT$  a set $S\subset \TT^c$ of compact objects is a set of
classical generators for the subcategory of compact objects $\TT^c\subset\TT$ if and only if
the set $S$ is a set of compact generators of $\TT.$

In our case we obtain that the set of free modules $\big\{A({\pp}),\;{\pp}\in L\big\}$ classically generates the category $\perf A$ and we
obtain the following proposition.

\begin{proposition}
Let $A=\bigoplus_{\pp\in L}A_{\pp}$ be an $L$\!--graded algebra. Suppose that $A_{\00}\cong\kk$ and there is a function $w:L\to\Z$ such that
$A_{\qq-\pp}=0$ when $w(\qq)\ge w(\pp)$ and $\qq\ne \pp.$
Then the set of free modules $\big\{A({\pp}),\;{\pp}\in L\big\}$ forms a full exceptional collection in the triangulated category
of perfect complexes $\perf A\subset \bD(\GrL  A).$ (Note that $w$ is not assumed to be linear).
\end{proposition}

It is frequently the case that  two non-isomorphic $L$\!--graded algebras $A$ and $B$ produce equivalent  categories
$\tA$ and $\tB$ and, hence, they have equivalent abelian categories of $L$\!--graded modules $\GrL  A$ and $\GrL  B.$

\begin{example}
{\rm
A simple example is given by $\Z$\!--graded algebras 
\[
S_2=\kk[x_1, x_2]=\kk\{x_1,x_2\}/\langle x_1x_2-x_2x_1\rangle
\] 
and 
\[
S_{2,q}=\kk\{y_1,y_2\}/\langle y_1y_2-qy_2y_1\rangle
\] 
for some $q\in \kk^*.$ It is not difficult to see that the categories $\tS_2$ and $\tS_{2,q}$
are equivalent. This can be done directly and would also follow from Corollary \ref{5.9} below.
}
\end{example}

Let us consider the polynomial algebra $\Sn=\kk\left[x_1,\ldots, x_n\right]$ with the maximal grading $L=\Z^n.$
The corresponding $\kk$\!--linear category $\tSn$ that is the category of full exceptional collection of
free modules $\big\{\Sn(\pp), \pp\in\Z^n\big\}$ has the following properties:
\begin{equation*}\label{ND}
\parbox{.915\textwidth}{
\begin{enumerate}
\item[(1)] the set of objects is isomorphic to
$\Z^n,$
\item[(2)]
$\Hom(\pp_1 ,\pp_2)\cong\kk,$ when $\pp_1\leq \pp_2,$
otherwise they are equal to
$0,$
\item[(3)] the composition
$
\Hom(\pp_2 ,\pp_3)\otimes \Hom(\pp_1 ,\pp_2) \stackrel{\sim}{\lto} \Hom(\pp_1 ,\pp_3)
$
is nondegenerate  for all $\pp_1\leq \pp_2\leq \pp_3.$
\end{enumerate}
}\leqno \raisebox{8pt}{$(\widetilde\Zz^n)$}
\end{equation*}

The main result of this  subsection is the following observation that
 the category $\tSn$ is uniquely defined by the properties $(\widetilde\Zz^n).$
\begin{proposition}\label{noncom}
Let $\tE$ be a category for which $(\widetilde\Zz^n)$ hold.
Then $\tE$ is equivalent to $\tSn.$
\end{proposition}
\begin{proof}
Let $\yY^{i}_{\pp}$ be some elements of $\Hom_{\tE}(\pp, \pp+\ee_i).$ Properties (2) and (3) of $(\widetilde\Zz^n)$
imply that there are invertible elements $ \Theta^{ij}_{\pp}\in\kk^*$ such that
$\Theta^{ij}_{\pp}=(\Theta^{ji}_{\pp})^{-1}$ and
\begin{equation*}
\yY^i_{\pp+\ee_j} \yY^j_{\pp}=\Theta^{ij}_{\pp} \yY^j_{\pp+\ee_i} \yY^i_{\pp}.
\end{equation*}
These elements are not independent and there are the following  basic Yang--Baxter type relations
\begin{equation}\label{YB}
\Theta^{jk}_{\pp+\ee_i}\Theta^{ik}_{\pp} \Theta^{ij}_{\pp+\ee_k}=\Theta^{ij}_{\pp}\Theta^{ik}_{\pp+\ee_j}\Theta^{jk}_{\pp}.
\end{equation}
These relations follow directly from the existence of two different ways
of passing from the composition $\yY^i_{\pp+\ee_k+\ee_j}\yY^j_{\pp+\ee_k}\yY^k_{\pp}$ to the composition $\yY^k_{\pp+\ee_j+\ee_i}\yY^j_{\pp+\ee_i}\yY^i_{\pp}.$
Indeed, we have

\begin{align*}
\begin{split}
\yY^i_{\pp+\ee_k+\ee_j}\yY^j_{\pp+\ee_k}\yY^k_{\pp}& = \Theta^{jk}_{\pp}\yY^i_{\pp+\ee_k+\ee_j}\yY^k_{\pp+\ee_j}\yY^j_{\pp}=
\Theta^{ik}_{\pp+\ee_j}\Theta^{jk}_{\pp}\yY^k_{\pp+\ee_j+\ee_i}\yY^i_{\pp+\ee_j}\yY^j_{\pp}=\\
& = \Theta^{ij}_{\pp}\Theta^{ik}_{\pp+\ee_j}\Theta^{jk}_{\pp}\yY^k_{\pp+\ee_j+\ee_i}\yY^j_{\pp+\ee_i}\yY^i_{\pp},
\end{split}
\\
\begin{split}
\yY^i_{\pp+\ee_k+\ee_j}\yY^j_{\pp+\ee_k}\yY^k_{\pp}& = \Theta^{ij}_{\pp+\ee_k}\yY^j_{\pp+\ee_k+\ee_i}\yY^i_{\pp+\ee_k}\yY^k_{\pp}=
\Theta^{ik}_{\pp} \Theta^{ij}_{\pp+\ee_k}\yY^j_{\pp+\ee_k+\ee_i}\yY^k_{\pp+\ee_i}\yY^i_{\pp}=\\
 & = \Theta^{jk}_{\pp+\ee_i}\Theta^{ik}_{\pp} \Theta^{ij}_{\pp+\ee_k}\yY^k_{\pp+\ee_j+\ee_i}\yY^j_{\pp+\ee_i}\yY^i_{\pp}.
\end{split}
\end{align*}

Let us make a change of variables and put $\xX^s_{\pp}=a^s_{\pp}\yY^s_{\pp},$ where
coefficients $a^s_{\pp}\in\kk^*$ are uniquely defined by the following recurrence relations:

\begin{itemize}
\item
$a^s_{\pp}=1$ for every $\pp$ such that $\pp(1)=\cdots=\pp(s-1)=0.$ In particular, $a^1_{\pp}=1$ for all $\pp\in \Z^n,$
\item
for every $\pp$ such that $\pp(1)=\cdots=\pp(r-1)=0$ for some $r<s$ we have $a^s_{\pp}=\Theta^{rs}_{\pp-\ee_r} a^s_{\pp-\ee_r}.$
In particular, $a^s_{\pp}=\Theta^{1s}_{\pp-\ee_1} a^s_{\pp-\ee_1}$ for all $\pp\in\Z^n.$
\end{itemize}

Let us show that $\xX^j_{\pp+\ee_k} \xX^k_{\pp}=\xX^k_{\pp+\ee_j} \xX^j_{\pp}$ for all $\pp\in\Z^n,$ i.e.  we have to check that
\begin{equation}\label{rel}
\Theta^{jk}_{\pp} a^j_{\pp+\ee_k} a^k_{\pp}=a^k_{\pp+\ee_j} a^j_{\pp}
\end{equation}
for all $\pp.$

We can assume that $j< k$ and
proceed by induction on $|\pp|=\sum_r |\pp(r)|.$
Let $i$ be such that $\pp(1)=\cdots=\pp(i-1)=0$ and $\pp(i)\ne 0.$
If $j\le i,$ then $a^j_{\pp+\ee_k}=a^j_{\pp}=1$ and $a^k_{\pp+\ee_j}=\Theta^{jk}_{\pp}a^k_{\pp}.$ Hence, the relation (\ref{rel}) holds.
Let $i< j < k.$ Assume that $\pp(i)>0.$ Then we have
\begin{multline*}
a^k_{\pp+\ee_j} a^j_{\pp}=\Theta^{ij}_{\pp-\ee_i} \Theta^{ik}_{\pp-\ee_i + \ee_j} a^k_{\pp-\ee_i + \ee_j} a^j_{\pp-\ee_i}=
\Theta^{ij}_{\pp-\ee_i} \Theta^{ik}_{\pp-\ee_i + \ee_j} \Theta^{jk}_{\pp-\ee_i}a^j_{\pp-\ee_i + \ee_k}
a^k_{\pp-\ee_i}\overset{(\ref{YB})}{=\joinrel=\joinrel=}\\
=\Theta^{jk}_{\pp}\Theta^{ik}_{\pp-\ee_i} \Theta^{ij}_{\pp-\ee_i+\ee_k}a^j_{\pp-\ee_i + \ee_k} a^k_{\pp-\ee_i}=\Theta^{jk}_{\pp}a^j_{\pp +
\ee_k} a^k_{\pp},
\end{multline*}
where the second equation holds by the induction hypothesis.
If now $\pp(i)<0$ we have the following sequence of equations
\begin{multline*}
a^j_{\pp+\ee_k} a^k_{\pp}= \Theta^{ji}_{\pp+ \ee_k} \Theta^{ki}_{\pp} a^j_{\pp+\ee_i + \ee_k} a^k_{\pp+\ee_i}=
\Theta^{ji}_{\pp+ \ee_k} \Theta^{ki}_{\pp}\Theta^{kj}_{\pp+\ee_i}  a^k_{\pp+\ee_i + \ee_j}
a^j_{\pp+\ee_i}\overset{(\ref{YB})}{=\joinrel=\joinrel=}\\
=\Theta^{kj}_{\pp} \Theta^{ki}_{\pp+\ee_j}\Theta^{ji}_{\pp}  a^k_{\pp+\ee_i + \ee_j} a^j_{\pp+\ee_i} =\Theta^{kj}_{\pp} a^k_{\pp+\ee_j}
a^j_{\pp},
\end{multline*}
where the second equation holds by the induction hypothesis as well.

Thus, we obtain $\xX^j_{\pp+\ee_k} \xX^k_{\pp}=\xX^k_{\pp+\ee_j} \xX^j_{\pp}$ for all $\pp\in\Z^n.$ Therefore, $\tE$ is equivalent to $\tSn.$
\end{proof}

\emph{Alternative proof.}
For each integer $m$ denote by $\overline m$  the object in the category that corresponds to $(m,\ldots,m)\in\Z^n.$
Pick nonzero elements $Y_m\in \Hom( \overline m, \overline{m+1}).$ We use compositions of $Y_m$ to define nonzero elements
$Y_{m_1,m_2} \in \Hom(\overline {m_1},\overline {m_2})$ for all $m_1\leq m_2,$ with $Y_{m,m}$ being identity.
These elements will automatically satisfy
\begin{equation}\label{assoY}
Y_{m_2,m_3} \circ Y_{m_1,m_2} = Y_{m_1,m_3}
\end{equation}
for all $m_1\leq m_2\leq m_3.$

For each $\pp\in \Z^n$ pick a nonzero element $U_{\pp}\in \Hom( \overline{\min (\pp)}, \pp).$
For each $\pp_1\leq \pp_2$ in $\Z^n$ pick the unique nonzero element $V_{\pp_1,\pp_2} \in \Hom(\pp_1,\pp_2)$ with the property
\[
V_{\pp_1,\pp_2} \circ U_{\pp_1} =  U_{\pp_2} \circ Y_{\min(\pp_1),\min(\pp_2)}.
\]
We claim that for any $\pp_1\leq \pp_2\leq \pp_3$ there holds $V_{\pp_2,\pp_3} \circ V_{\pp_1,\pp_2} =V_{\pp_1,\pp_3}.$

Indeed, we have
\begin{multline*}
V_{\pp_2,\pp_3} \circ  V_{\pp_1,\pp_2} \circ U_{\pp_1} =  V_{\pp_2,\pp_3} \circ U_{\pp_2} \circ Y_{\min(\pp_1),\min(\pp_2)} =
U_{\pp_3} \circ Y_{\min(\pp_2),\min(\pp_3)} \circ Y_{\min(\pp_1),\min(\pp_2)}=\\ 
= U_{\pp_3} \circ Y_{\min(\pp_1),\min(\pp_3)}
\end{multline*}
by definition of $V$ and \eqref{assoY}.
Together with
$
V_{\pp_1,\pp_3} \circ U_{\pp_1}  =
 U_{\pp_3} \circ Y_{\min(\pp_1),\min(\pp_3)}$
this implies
\[
 V_{\pp_2,\pp_3} \circ  V_{\pp_1,\pp_2} \circ U_{\pp_1}=V_{\pp_1,\pp_3} \circ U_{\pp_1}
\]
which then leads to the desired associativity of $V.$ The choices of $V$ allow one to construct the equivalence to $\tSn.$
\hfill{$\Box$}

\medskip
Let $\theta=(\theta^{ij})\in \oM_n(\kk^*)$ be a multiplicatively  skew-symmetric matrix, i.e. $\theta^{ji}=(\theta^{ij})^{-1}$ and
$\theta^{ii}=1.$
Consider a noncommutative algebra $\Snt=\kk_{\theta}[y_1, \ldots, y_n]$ that is a noncommutative toric deformation of
the polynomial algebra $\Sn=\kk[x_1,\ldots, x_n].$ The algebra $\Snt$ is generated by the variables $y_1, \ldots, y_n$
with commutation relations of the form $y_i y_j=\theta^{ij} y_j y_i.$ It can be considered as $\Z^n$\!--graded algebra.
Let us take the abelian category $\GrZ \Snt$ of all right $\Z^n$\!--graded modules and
the category $\tSnt\subset \GrZ \Snt$ formed by the full exceptional collection of free modules $(\Snt(\pp), \pp\in\Z^n).$
The following statement is a corollary of Proposition~\ref{noncom}.
\begin{corollary}\label{5.9}
Let $\Snt=\kk_{\theta}[y_1,\ldots, y_n]$ be the algebra with relations $y_i y_j=\theta^{ij} y_j y_i$ for some
multiplicatively skew-symmetric matrix $\theta=(\theta^{ij})\in \oM_n(\kk^*).$ Then the category $\tSnt$ is equivalent to the category
$\tSn,$ and the abelian categories $\GrZ \Snt$ and $\GrZ \Sn$ are equivalent too.
\end{corollary}
\begin{proof}
The category $\tSnt$ satisfies properties $(\widetilde\Zz^n)$ and, hence, by Proposition~\ref{noncom} it is equivalent to $\tSn.$
Since the abelian category $\GrZ \Snt$ is equivalent to the abelian category $\Mod \tSnt,$ we obtain that
$\GrZ \Snt$ is equivalent to $\GrZ \Sn.$
\end{proof}

\begin{remark}
More generally, we don't expect any interesting noncommutative deformations of $[U_\Sigma/\Gm^n]$ 
for a simplicial complex $\Sigma,$ in contrast to the non-equivariant setting.
\end{remark}

%\newpage

\end{document}